\documentclass[a4 paper, 11 pt]{article}
\usepackage{a4wide, amsmath, amssymb, mathtools, yfonts, amsrefs}
\usepackage{comment, bbold, enumerate}
\usepackage{tikz}
\usepackage{tikz-cd}
\usepackage[all]{xy}
\usepackage[utf8]{inputenc}
\usepackage{amsthm, mathrsfs}
\usepackage[english]{babel}
\usepackage{hyperref}
\usepackage{authblk}
\usepackage[OT2,T1]{fontenc}
\DeclareSymbolFont{cyrletters}{OT2}{wncyr}{m}{n}
\DeclareMathSymbol{\Sha}{\mathalpha}{cyrletters}{"58}

\numberwithin{equation}{section}
\newtheorem{lemma}{Lemma}[section]
\newtheorem{theorem}[lemma]{Theorem}
\newtheorem{proposition}[lemma]{Proposition}
\newtheorem{corollary}[lemma]{Corollary}
\newtheorem{notation}[lemma]{Notation}

\theoremstyle{remark}
\newtheorem{remark}{Remark}
\newtheorem*{remark*}{Remark}

\theoremstyle{definition}

\newcommand{\N}{\mathbb{N}}
\newcommand{\Z}{\mathbb{Z}}
\newcommand{\Q}{\mathbb{Q}}

\newcommand{\R}{\mathbb{R}}

\newcommand{\GL}{\mathrm{GL}}

\newcommand\Aut{\text{Aut}}

\newcommand\Gal{\mathrm{Gal}}

\newcommand\Disc{\mathrm{Disc}}
\newcommand\Stab{\mathrm{Stab}}
\newcommand\Res{\mathrm{Res}}

\newcommand\ord{\mathrm{ord}}

\renewcommand{\pmod}[1]{\,(\mathrm{mod}\,#1)}

\title{\vspace{-\baselineskip}\sffamily\bfseries Asymptotics for $6$-torsion and $D_6$-extensions}
\author[1]{P.~Koymans\thanks{Mathematisch Instituut, Universiteit Utrecht, Postbus 80.010, 3508 TA Utrecht, The Netherlands, p.h.koymans@uu.nl}}
\author[2, 3]{R.J.~Lemke Oliver\thanks{Department of Mathematics, University of Wisconsin-Madison, Madison, WI 53706, USA, lemkeoliver@wisc.edu}\thanks{Department of Mathematics, Tufts University, Medford, MA 02155, USA, robert.lemke oliver@tufts.edu}}
\author[4]{E.~Sofos\thanks{Universit\`a di Roma Tor Vergata, Dipartimento di Matematica, 00133, Rome, sofos@mat.uniroma2.it}}
\author[5]{F.~Thorne\thanks{Department of Mathematics, University of South Carolina, 1523 Greene Street, Columbia, SC
29208, thorne@math.sc.edu}}
\affil[1]{Utrecht University}
\affil[2]{University of Wisconsin-Madison}
\affil[3]{Tufts University}
\affil[4]{Universit\`a di Roma Tor Vergata}
\affil[5]{University of South Carolina}
\date{\today}

\begin{document}
\maketitle

 
\begin{abstract}
We prove a composite case of the Cohen--Lenstra--Gerth heuristics. Specifically, we establish  an asymptotic for the average $6$-torsion of the class group of quadratic number fields. We also prove Malle's conjecture for Galois $D_6$-extensions.
\end{abstract}
\section{Introduction}
\subsection{Main results}
For a discriminant $D$ and $n\in \mathbb N$, write $h_n(D) := \# \mathrm{Cl}(\Q(\sqrt{D}))[n]$. The main objective of this paper is to prove the following asymptotic formula for the average of $h_6(D)$, as $D$ ranges over discriminants of quadratic fields with specified signature. For imaginary quadratic fields we obtain:

\begin{theorem}
\label{thm:main-}
We have
\[
\sum_{\substack{-X < D < 0 \\ D \textup{ fundamental}}} h_6(D) = 
\frac{3}{\pi^2} X (\log X) \prod_{\substack{p=2\\p \ \mathrm{ prime}}}^\infty \left(1 + \frac{1}{p + 1}\right) \left(1 - \frac{1}{p}\right) 
+ O(X (\log \log X)^{7/2}).
\]
\end{theorem} 

This is the first instance in which the Cohen--Lenstra--Gerth heuristics (see \cites{CohenLenstra1,CohenLenstra2,Gerth}) have been established for $h_n$ when the number of distinct prime factors of $n$ is strictly larger than $1$.
These heuristics predict that for each $n\geq 2$ there is a constant $c_n>0$ such that 
$$
\sum_{\substack{-X < D < 0 \\ D \textup{ fundamental}}} h_n(D) \sim
\begin{cases}c_n X, &\text{if } n \text{ is odd,} \\
c_n X \log X, &\text{if } n \text{ is even}
\end{cases}
$$
as $X \to \infty$. In particular, Theorem \ref{thm:main-} proves a case of the Cohen--Lenstra--Gerth heuristics for class group torsion in imaginary quadratic fields. Such results seem hard to come by. To date, asymptotics for $h_n$ have only been proved for $n = 2$ (elementary; see \S \ref{ss6Main}), $n = 3$ (Davenport--Heilbronn \cite{DH}); $n = 4$ (Fouvry--Kl\"uners \cite{FKIMRN}); and $n$ an arbitrary power of $2$ (Smith \cites{Smith,Smith2}). For more general number fields, the average $2$-torsion of cubic fields was determined by Bhargava--Varma \cite{BV} and the average $3$-torsion of certain $2$-extensions was determined by Lemke Oliver--Wang--Wood \cite{LOWW}. Returning to quadratic fields, upper bounds of the correct order of magnitude are available for all $n$ of the shape $3 \cdot 2^k$ by work of Koymans--Pagano--Sofos \cite{KPS}.

It is widely expected that different parts of class groups are \emph{uncorrelated} but our result is the first to rigorously prove this in a concrete case. More precisely, for $n\in \mathbb N$ and $X> 3$ let 
$$
\mathbb{E}_n(X) = \frac{1}{\#\{D\in (-X,0): D \textup{ fundamental} \}}
\sum_{\substack{-X < D < 0 \\ D \textup{ fundamental}}} h_n(D) .
$$ 
For fixed $D$, when $n, m$ are coprime, we have $h_{nm}(D) = h_n(D) h_m(D)$. Consequently, we might expect that $\mathbb{E}_{nm}(X)$ is asymptotic to $\mathbb{E}_n(X) \mathbb{E}_m(X)$ as $X \to\infty$.  It follows from Theorem~\ref{thm:main-} that this is the case when $n=2$ and $m=3$: 

\begin{corollary}
We have
$$
\mathbb{E}_6(X) = \mathbb{E}_2(X) \mathbb{E}_3(X)\left(1+o_{X\to\infty}(1)\right).
$$
\end{corollary}

For real quadratic fields, we obtain the averages of both $h^+_6(D)$ and $h_6(D)$:

\begin{theorem}
\label{thm:main+}
We have
\[
\sum_{\substack{0 < D < X \\ D \textup{ fundamental}}} h^+_6(D) = 
\frac{2}{\pi^2} X(\log X) \prod_{\substack{p=2\\p \ \mathrm{ prime}}}^\infty \left(1 + \frac{1}{p + 1}\right) \left(1 - \frac{1}{p}\right) 
+ O(X (\log \log X)^{7/2}),
\]
and
\[
\sum_{\substack{0 < D < X \\ D \textup{ fundamental}}} h_6(D) \sim
\frac{1}{\pi^2} X(\log X) \prod_{\substack{p=2\\p \ \mathrm{ prime}}}^\infty \left(1 + \frac{1}{p + 1}\right) \left(1 - \frac{1}{p}\right).
\]
\end{theorem} 

As $h^+_6(D) = 2h_6(D)$ whenever $D$ has any prime divisor $\equiv 3 \pmod{4}$, the result for $h_6(D)$ will be easy to conclude from that for $h^+_6(D)$. 

Our final main result, whose proof uses a technical generalization of the above results to handle certain local conditions, 
establishes Malle's conjecture for Galois $D_6$-extensions, where $D_6$ is the dihedral group of order $12$. Let 
	\[
		\mathcal{F}(X;D_6)
			:= \{ K/\mathbb{Q} \textnormal{ Galois} : |\mathrm{Disc}(K)| \leq X, \mathrm{Gal}(K/\mathbb{Q}) \simeq D_6\}.
	\]

\begin{theorem}
\label{tMalle}
We have
\[
\# \mathcal{F}(X;D_6)
		=
C X^{1/6} (\log X)^2 + O\big( X^{1/6} (\log X) (\log \log X)^{9/2} \big),
\]
where $C$ is given by 
$$  \frac{1}{2^2 3^7} \left( \frac{23}{8} + \frac{1}{2^{4/3}} + \frac{1}{2^{5/3}}\right)\left(\frac{14}{9} + \frac{1}{3^{4/3}} + \frac{1}{3^{5/3}}\right)
 \prod_{\substack{p=5\\p \ \mathrm{ prime}}}^\infty \left(1-\frac{1}{p}\right)^3\left(1 + \frac{3}{p} + \frac{1}{p^{4/3}} + \frac{1}{p^{5/3}}\right).
$$
\end{theorem}

We check in \S\S \ref{ssLS1}-\ref{ssLS2} that the leading constant in Theorem \ref{tMalle} agrees with a recent prediction by Loughran--Santens \cite[Conjecture 1.3(1)]{LS}. An asymptotic for sextic non-Galois $D_6$-extensions is given in \cite{MTTW} by Masri--Thorne--Tsai--Wang.  Additionally, an asymptotic formula for the number of Galois $D_4$-extensions was recently given by Shankar--Varma \cite{ShankarVarma}.
One important feature of Theorem~\ref{tMalle} is that, in its regular degree $12$ representation, the group $D_6$ is not \emph{concentrated} in the language of \cite{ALOWW}.  Prior to our work, the only known cases of Malle's conjecture for nonabelian unconcentrated groups were for $S_3$-cubics (Davenport--Heilbronn \cite{DH}), $S_4$-quartics and $S_5$-quintics (Bhargava \cites{Bhargava-Quartic,Bhargava-Quintic}), Galois $S_3$-extensions (Bhargava--Wood \cite{BhargavaWood}, Belabas--Fouvry \cite{BelabasFouvry}), and the aforementioned work of Shankar--Varma on Galois $D_4$-extensions.  In particular, such results are almost as hard to come by as cases of the Cohen--Lenstra--Gerth heuristics.

Among several technical inputs into the proof, we highlight an improved level of distribution result for cubic fields and $3$-torsion, with respect to divisibility of the discriminant. The following is an example of what we prove:

\begin{theorem}\label{thm:lod_sample}
There is a multiplicative function $f(d)$ so that
\[
\sum_{\substack{d < X^{1/2} (\log X)^{-5} \\ \textnormal{ squarefree}}} \bigg| \sum_{\substack{|D| \leq X \\ d \mid D}} h_3(D) - \frac{10}{\pi^2} f(d) X\bigg| = o(X).
\]
\end{theorem}

Here the upper limit $X^{1/2} (\log X)^{-5}$ on $d$ is called a {\itshape level of distribution}. We refer to Theorem \ref{thm:ld} for a precise statement which generalizes Theorem \ref{thm:lod_sample} simultaneously in several directions: to allow for counting cubic fields of discriminant $D$, in place of $h_3(D)$; to restrict to only positive or negative discriminants; to allow `local conditions' within a certain range; to save a power
of $\log X$ in the error term; and, in each error estimate to allow $X$ to vary with $d$.

More precise asymptotics for $\sum_D h_3(D)$ and related counts would include secondary terms of order $X^{5/6}$, as proved in \cite{BST} and \cite{TT_rc}. As these secondary terms decay with $d$, they may be subsumed into the error term above. 

A statement similar to Theorem \ref{thm:lod_sample} was proved by Bhargava--Taniguchi--Thorne \cite{BTT}, with a level of distribution of $X^{1/2 - \epsilon}$. The improvements of Theorems~\ref{thm:lod_sample} and Theorem~\ref{thm:ld} over this are rather modest, but they turn out to be critical to the proofs of our other theorems. Another crucial technical tool is the recent work by 
La Bret{\`e}che--Tenenbaum~\cite{brete2}
and  Koukoulopoulos--Tao \cite{koukoutao}   
on the average of the Hooley Delta function, defined by  
$$\Delta(n)=\sup_{u\geq 1 }\#\{d\in \N\cap [u,\mathrm e u]: d\mid n\}, \ \ \ n \in \mathbb N.$$ 
This is used to handle moduli in the interval  $[X^{1/2}(\log X)^{-5}, X^{1/2}]$. The use of Hooley's Delta function  appears to be new in the context of the Cohen--Lenstra heuristics and field counting.

\subsection{Overview of this paper}
\usetikzlibrary{positioning,arrows.meta,quotes}

We now sketch the proof of Theorem \ref{thm:main-} and give an overview of the rest of this paper.

In order to obtain an asymptotic for $h_6(D)$, recall that by Gauss's genus theory, we have
\[
h^+_2(D) = 2^{\omega(D) - 1},
\]
where $\omega(D)$ is the number of prime divisors of $D$. When $D < 0$ we have $h_2(D) = h^+_2(D)$, and when $D > 0$ we have 
\begin{equation}
\label{eh2D}
h_2(D) = 
\begin{cases}
h^+_2(D) &\text{if } p \mid D \Rightarrow p \equiv 1, 2 \bmod 4, \\
\frac{1}{2} h^+_2(D) &\text{otherwise.}
\end{cases}
\end{equation}
Since $h_2(D)$ is closely connected to the divisor function $\tau(D)$, we essentially have 
$$
\sum_{D \leq X} h_6(D) \approx \sum_{D \leq X} \tau(D) h_3(D) \approx 2 \sum_{d \leq \sqrt{X}} \sum_{\substack{D \leq X \\ d \mid D }} h_3(D)
$$
by Dirichlet's hyperbola method. If a level of distribution of $X^{1/2}$ could be proved for $h_3(D)$, then our task would have been straightforward.

In \cites{4author, 4authorapplications}, Chan--Koymans--Pagano--Sofos used a sieve argument going back to Erd\H{o}s \cite{Erdos2}, together with the level of distribution of $X^{1/2 - \epsilon}$ proved in \cite{BTT}, to obtain the upper bound 
\begin{equation}\label{eq:prev_bound}
\sum_{|D| \leq X} h_6(D) \ll X \log X.
\end{equation}
In fact, any level of distribution will do for this argument. Our current knowledge is aptly summarized by the following diagram.
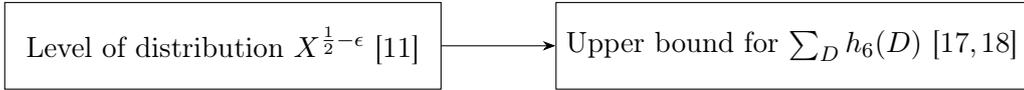
\begin{figure}[!h]
\centering
\begin{tikzpicture}
 [block/.style={draw,minimum width=#1,minimum height=2em},
 block/.default=15em,high/.style={minimum height=3em},auto,
 node distance=15mm, 
 >=Stealth]
 \node[block=15em,high,right] (n1) {Level of distribution $X^{\frac{1}{2} - \epsilon}$ \cite{BTT}};
 \node[block=15em,high,right=of n1] (n2) {Upper bound for $\sum_D h_6(D)$ \cites{4author, 4authorapplications}};
 \draw[->] (n1) -- (n2);
\end{tikzpicture}
\caption{Current state of the art}
\end{figure}

This upper bound is itself an ingredient in Theorems \ref{thm:lod_sample} and \ref{thm:ld}, i.e.~in the improvement of the level of distribution to $X^{1/2} (\log X)^{-5}$. The reason for this seemingly circular reasoning can be seen in the details of our proof of Theorem \ref{thm:ld}, which we give in \S \ref{sLevel}; in brief, after applying Fourier analysis, one obtains error terms that have structural similarities to the main terms. This leaves the terms with $X^{1/2}/(\log X)^5 \leq d \leq X^{1/2}$, and it turns out that the sieve comes to the rescue again for this range. Fortunately, it is enough to obtain a (close to) sharp upper bound in this range, and this is done in \S \ref{ss6Trick} by the sieve techniques from \cites{4author, 4authorapplications} combined with work of Koukoulopoulos--Tao \cite{koukoutao} and La Bret{\`e}che--Tenenbaum~\cite{brete2} on the Hooley Delta function. Once more, the sieve uses as a vital input that $h_3(D)$ has a level of distribution.

\begin{figure}[!h]
\centering
\begin{tikzpicture}
 [block/.style={draw,minimum width=#1,minimum height=2em},
 block/.default=15em,high/.style={minimum height=3em},auto,
 node distance=25mm, 
 >=Stealth]
 \node[block=16em,high,right] (n1) {Level of distribution $X^{\frac{1}{2} - \epsilon}$ \cite{BTT}};
 \node[block=16em,high,right=of n1] (n2) {Upper bound for $\sum_D h_6(D)$ \cites{4author, 4authorapplications}};
 \node[block=16em,below of = n1] (n3) {Level of distribution $X^{1/2}/(\log X)^A$};
 \node[block=16em,below of = n2] (n4) {Asymptotic for $\sum_D h_6(D)$};
 \draw[->] (n1) -- (n2) node[midway] {\tiny{Sieve}} ;
 \draw[->] (n2) -- (n3) node[midway, yshift = 0.02in, xshift = -0.9in] {\tiny{\S \ref{sLevel}}};
 \draw[->] (n3) -- (n4) node[midway, above=0.15in] {\tiny{\S \ref{ss6Main}}} node[midway, above=0.05in] {\tiny{Small divisors}};
 \draw[->] (n1) -- (n4) node[midway, yshift = -0.05in, xshift = 0.4in] {\tiny{\S \ref{ss6Trick}}} node[midway, yshift = -0.15in, xshift = 0.4in] {\tiny{Large divisors}};
\end{tikzpicture}
\caption{Our argument}
\end{figure}
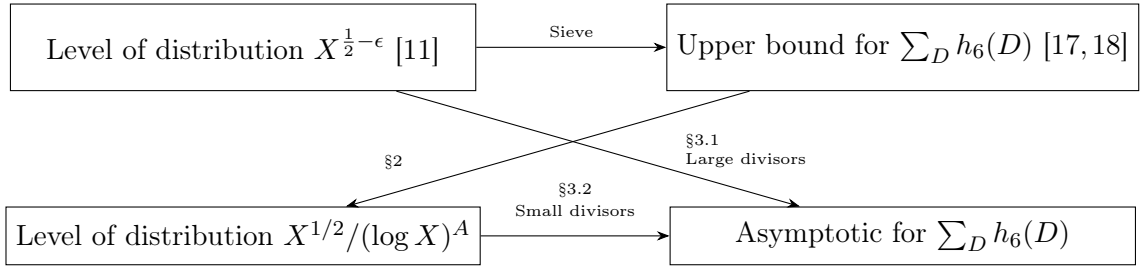
We then prove Theorems \ref{thm:main-} and \ref{thm:main+} in \S \ref{ss6Main}. Our arguments are flexible enough to obtain what amounts to an average of $\sum_D h_6(D)$ with local conditions; we prove a technical result along these lines in \S \ref{sslocal}. This will allow us to also prove Malle's conjecture for Galois $D_6$-extensions, which we carry out in \S \ref{sec:malle}. Once more, the sieve comes to the rescue in an essential way for various vital tail estimates.


\subsection*{Acknowledgements}
We would like to thank Stephanie Chan, Daniel Loughran, Carlo Pagano, Tim Santens, and Takashi Taniguchi for helpful comments and suggestions.

PK gratefully acknowledges the support of Dr. Max R\"ossler, the Walter Haefner Foundation and the ETH Z\"urich Foundation. PK also acknowledges the support of the Dutch Research Council (NWO) through the Veni grant ``New methods in arithmetic statistics''.
RJLO was partially supported by a National Science Foundation grant, DMS-2200760, and by the Office of the Vice Chancellor for Research at the University of Wisconsin-Madison with funding from the Wisconsin Alumni Research Foundation. FT was partially supported by the National Science Foundation grants DMS-2101874 and DMS-2503694 and by the Lodha Mathematical Sciences Institute. 

\section{The level of distribution of cubic fields}
\label{sLevel}
In this section we prove the level of distribution estimate of Theorem \ref{thm:lod_sample}, in a more general form. As mentioned previously, this is 
a modest improvement over a similar
estimate proved by Bhargava--Taniguchi--Thorne \cite{BTT}, which in turn built on the previous works \cites{DH, belabas, BBP, BST, TT_rc}).

The proof consists essentially of a tweaking of the argument in \cite{BTT} and its technical core. 
Here, we will not duplicate all the exposition
from \cite{BTT} but instead give a briefer overview which we hope will be reasonably comprehensible, while eliminating some complications
which are extraneous to the case considered here.

We write 
\[
	N_3(X) := \# \{ K \ : \ [K : \Q] = 3, \ |\Disc(K)| < X \}
\]
and, for each integer $d$,
\begin{equation}\label{eq:def_n3xd}
N_3(X, d) := \# \{ K : \ [K : \Q] = 3, \ |\Disc(K)| < X, \ d \mid \Disc(X) \}.
\end{equation}
Here and throughout, we count cubic fields up to isomorphism.
We define an error term $E(X, d)$ by writing
$$
N_3(X, d) = M \frac{g(d)}{d} X + E(X, d),
$$
where $g(d)$ is multiplicative with
$$
M := \frac{1}{3 \zeta(3)}, \ \ \ \ g(p) := 
\frac{p^2 + p}{p^2 + p + 1}.
$$
We also write
\begin{equation}\label{def:exd_alt}
N^{\ast}_3(X, d) = M^* \frac{g^*(d)}{d} X + E^\ast(X, d),
\end{equation}
where the asterisk denotes that we count only those cubic fields whose discriminants are fundamental, and
\begin{equation}\label{def:omega1}
M^* := \frac{2}{\pi^2}, \ \ \ \ g^*(p) := 
\frac{p}{p + 1}.
\end{equation}
We will append $+$ or $-$ to any of the above notations when counting positive or negative discriminants only, except for $g(d)$ and $g^*(d)$ which 
do not change with this restriction. We have
\[
M^+ = \frac{1}{4} M, \ \ 
M^- = \frac{3}{4} M, \ \ 
M^{*,+} = \frac{1}{4} M^*, \ \ 
M^{*,-} = \frac{3}{4} M^*. \ \ 
\]
Finally, for a prime $p$, we define a {\itshape local specification} $\pmod p$
to be a subset $\Sigma_p$ of the cubic \'etale algebras over $\Q_p$; the number of such algebras is finite and bounded uniformly over $p$.
We say that a cubic field $K$ satisfies $\Sigma_p$ if $K \otimes_{\Q} \Q_p \in \Sigma_p$. 
Similarly, for a squarefree integer $m$, we define a local specification $\Sigma_m \pmod m$ to be a local specification for each prime dividing $m$.

We write $N_3(X, d, \Sigma_m)$ for the count in \eqref{eq:def_n3xd} where it is additionally required that each $K$ being counted satisfy $\Sigma_m$, $E(X, d, \Sigma_m)$ for the associated error term, and analogously for the other notations defined above. In the case of $N_3^*(X, d, \Sigma_m)$, for a cubic field $K$, we reinterpret `$K$ has fundamental discriminant' to mean that `$K \otimes_{\Q} \Q_p$ is not totally ramified over $\Q_p$ for all primes $p \nmid m$'. In particular, depending on $\Sigma_m$, $N_3^*(X, d, \Sigma_m)$ may count fields with discriminant divisible by $p^2$ for $p \mid m$.


See \cites{TT_rc, BTT} for a complete description of the main terms in such counts.

\smallskip

We arrive now at our main result:

\begin{theorem}\label{thm:ld}
We have for each $B, k \geq 0$ that 
\begin{equation}
\label{eq:lod2}
\sum_{m < (\log X)^B} \sum_{\Sigma_m \pmod{m}} \\ \sum_{\substack{d < X^{1/2} (\log X)^{-5-9B-5k} \\ d \textnormal{ squarefree}}} |E(X_{d,m}, d, \Sigma_m)| = O_{B, k}(X (\log X)^{-k}),
\end{equation}
where:
\begin{itemize}
\item The sum over $\Sigma_m$ is over all local specifications to moduli bounded by $(\log X)^B$. If $B = 0$, by abuse of notation the sum is to be understood as over local specifications $\pmod{6}$ only.
\item The $X_{d, m}$ are real numbers bounded by $X$ for each $d$ and $m$, satisfying 
for some constant $C > 0$ that
\[
d' \in (d/2, 2d), \ m' \in (m/2, 2m) \ \ \Longrightarrow \ \ X_{d', m'} \in (X_{d, m}/C, CX_{d, m}).
\]
\end{itemize}
The same also holds with $E^*(X_{d, m}, d, \Sigma_m)$ in place of $E(X_{d, m}, d, \Sigma_m)$, as defined in \eqref{def:exd_alt}, and/or when counting fields of positive or negative discriminant only.
\end{theorem}

\begin{remark*}
	For the remainder of this section, we allow implied constants to depend on the parameters $B$ and $k$; we omit them in our big-oh, $\ll$, and similar notation.
	We do so similarly for any implied constants in this section that depend on an arbitrary $\epsilon>0$, and we allow the value of $\epsilon$ to change from one instance to the next. \end{remark*}

With additional effort, the power of $\log X$ could be improved somewhat, but obtaining a level of distribution of $X^{1/2}$ or greater 
seems quite difficult.  This appears closely related to proving a level-aspect subconvexity bound for certain Shintani zeta functions.  There is recent work of Hough and Lee \cite{HL} establishing
$t$-aspect subconvexity for Shintani zeta functions, but we are unaware of any level-aspect results.

\medskip
Our proof is a modification of that in \cite{BTT}, and comprises the remainder of this section. We will give a brief introduction to the relevant background; for the reader's convenience we consistently refer to \cite{BTT}, but note that much of this is in the earlier literature.

Let $V(\Z)$ be the lattice of integral binary cubic forms, equipped with the natural action of $\GL_2(\Z)$. By the Delone--Faddeev correspondence \cite{DF}, originally due to Levi \cite{levi}, the set of $\GL_2(\Z)$-orbits on $V(\Z)$ is in one-to-one correspondence with the set of cubic rings.
For each triple of squarefree integers $q, d, m$ and local specification $\Sigma_m \pmod{m}$, 
we define a $\GL_2(\Z)$-invariant function $\Phi_{d,q,m}\colon V(\mathbb{Z}) \to \mathbb{R}$ as the characteristic function of a set given explicitly below; 
from this, we then define {\itshape Shintani zeta functions}
\begin{equation}\label{eq:shintani}
\xi^{\pm}(s, \Phi_{d, q, m}) := \sum_{\substack{x \in \GL_2(\Z) \backslash V(\Z) \\ \pm \Disc(x) > 0}} \frac{1}{|\Stab(x)|} \Phi_{d, q, m}(x) |\Disc(x)|^{-s}
= \sum^{\ast} \frac{1}{|\Aut(R)|} |\Disc(R)|^{-s}
\end{equation}
where the star on the second summation indicates that the cubic rings $R$ run only over those corresponding to $x \in V(\Z)$ in the support of $\Phi_{q,d,m}$ whose discriminant has sign matching that of the superscript $\pm$ in $\xi^{\pm}$. 

As mentioned above, the function $\Phi_{d, q, m}(x)$ is the characteristic function of a $\GL_2(\Z)$-invariant subset of $V(\Z)$.  The definition of this subset depends on whether we are aiming to prove Theorem~\ref{thm:ld} with $E(X_{d,m},d,\Sigma_m)$ (as in \eqref{eq:lod2}) or its variant with $E^*(X_{d,m},d,\Sigma_m)$.  In the first case, the underlying set consists of those $x \in V(\Z)$ such that:

\begin{enumerate}[(i)]
	\item the discriminant of $x$ is divisible by $d$;
	\item  $x$ is `nonmaximal at $q$',
		i.e.~the corresponding cubic ring $R$ satisfies that $R \otimes_\Z \Z_p$ is contained in a larger cubic ring over $\Z_p$ for {\itshape each} prime dividing $q$; and
	\item the cubic ring $R$ corresponding to $x$ satisfies $\Sigma_m$. (Although we do not indicate it explicitly in the notation, this condition
			depends on the choice of $\Sigma_m$ and not just $m$.)
\end{enumerate}
If we are instead aiming to prove Theorem~\ref{thm:ld} for $E^*(X_{d,m},d,\Sigma_m)$, the set is the same, except that in place of (ii) we demand for each prime $p \mid q$ that either $x$ is nonmaximal at $p$ or that $x$ has a triple root $\pmod{p}$.
By the Davenport--Heilbronn correspondence and further discussion in \cite[\S 6.2]{TT_rc}, in either case, these conditions depend only on $x \pmod{12dq^2m^2}$.
When we omit any of the subscripts from $\Phi_{d, q, m}$ below, this indicates that the corresponding condition is omitted. Also note that $d, q, m$ are not required to
be coprime.

Work of Shintani \cite{shintani}, F. Sat\={o} \cite{f_sato}, Datskovsky and Wright \cites{DW1, DW2}, and Taniguchi--Thorne \cite{TT_L} 
establishes: analytic continuation for these zeta functions, holomorphic apart from simple poles at $s = 1$ and $s = 5/6$; explicit formulas
for their residues; and a functional equation of the form
$$
\Delta(1-s)\cdot T\cdot \xi(1-s, \Phi_{d, q, m} )= (12d q^2m^2)^{4s} \left(\begin{smallmatrix}3&0\\0&-3\end{smallmatrix}\right) \cdot \Delta(s)\cdot T\cdot \xi(s, \widehat{\Phi_{d, q, m}}),
$$
where
$T=\left(\begin{smallmatrix}\sqrt3&1\\\sqrt3&-1\end{smallmatrix}\right)$,
$\xi(s, \Theta):=\left(\begin{smallmatrix}\xi^+(s, \Theta)\\\xi^-(s, \Theta)\end{smallmatrix}\right)$, and $\Delta(s)$ is a diagonal matrix of gamma factors. Standard
methods of Landau \cites{Landau1912, Landau1915}, as further extended by Chandrasekharan and Narasimhan \cite{CN} and Lowry-Duda--Taniguchi--Thorne \cite{LDTT}, then
yield partial sum estimates for the zeta functions in question. 

The Shintani zeta functions \eqref{eq:shintani} count both reducible and irreducible rings, and they incorporate a weighting of $|\Aut(R)|^{-1}$.
When $R$ is maximal with $\Disc(R) \neq 0$ we have $|\Aut(R)| \leq 6$, and $|\Aut(R)| = 1$ when $R$ is the maximal order of an $S_3$-cubic field. 

We define $N'(X_{d, m}, d, \Sigma_m)$, $E'(X_{d, m}, d, \Sigma_m)$ and their analogues with
respect to the above weighting, and including the reducible rings. Then a preliminary reduction establishes that if $E'(X_{d, m}, d, \Sigma_m)$ satisfies the level of distribution result \eqref{eq:lod2}, then 
$E(X_{d, m}, d, \Sigma_m)$ does as well, and analogously for the other counting functions defined above.
This amounts to counting $C_3$-cubic fields and $C_2$-quadratic fields with adequate error terms, which is easier than the $S_3$-cubic field
problem. The details are carried out in \S 4 and \S 8 of \cite{BTT} and carry over without change. 

We will also apply \eqref{eq:prev_bound}, which was proved in the predecessor paper \cite[Theorem 1.3]{4authorapplications} of Chan--Koymans--Pagano--Sofos, and which we state
equivalently as follows:

\begin{lemma}\label{eq:divisor_bound}
We have that
\[
\sum_{|D| < Y} h_3(D) \tau(D) \ll Y \log Y,
\]
where the sum is over fundamental discriminants, and $\tau(\cdot)$ denotes the divisor function.
\end{lemma}

We now broaden this result to include all cubic rings, and also obtain a variant that will be needed in a tail estimate. From now on we shall adopt the convention that $\tau(0) := 0$.

\begin{lemma}\label{lem:divisor_bound_1}
With the same notation, we have 
$$
\sum_{|\Disc(x)| < Y} \tau(|\Disc(x)|) \ll Y \log Y,
$$
where the sum is over $\GL_2(\Z)$-equivalence classes of binary cubic forms, or equivalently (by Delone--Faddeev) over cubic rings.

Moreover, for every squarefree integer $q$ we have that
\begin{equation}
\label{eq:divisor_bound_2}
\sum_{\substack{|\Disc(x)| < Y \\ q^2 \mid \Disc(x)}} \tau(|\Disc(x)|) \ll \frac{1}{q^{2 - \epsilon}} Y \log Y.
\end{equation}
\end{lemma}

\begin{proof}
By work of Datskovsky and Wright \cite[p. 30]{DW2}, we have the coefficient-wise inequality of formal Dirichlet series
\begin{align*}
\sum_{\Disc(R) \neq 0} |\Disc(R)|^{-s} & \leq
\sum_{K} |\Disc(K)|^{-s} \cdot \zeta(2s)^3 \zeta(4s) \zeta(6s - 1) \\
& \leq
\sum_{K} |\Disc(K)|^{-s} \cdot \zeta(2s)^4 \zeta(6s - 1),
\end{align*}
where the sum on the left is over cubic rings, and the sum on the right is over fields $K$ of degree at most $3$. 
Moreover, for each fundamental 
discriminant $D$ and each $n \geq 1$, write $a_D(n)$ for the number of fields $K$ of degree $\leq 3$ of discriminant $D n^2$.  
Class field theory implies there is an absolute constant $C>0$ so that
\begin{align*}
\sum_n a_D(n) n^{-s} & \leq C a_D(1) \zeta(s)^4,
\end{align*}
as proved by Datskovsky and Wright \cite[Lemmas 6.1 and 6.2]{DW3}. 
Putting these together with the formula $a_D(1) = \frac{ h_3(D) + 1}{2}$ (not forgetting the unique degree $1$ or $2$ field of discriminant $D$), we obtain
\[
\sum_{\Disc(R) \neq 0} |\Disc(R)|^{-s} \leq C
\zeta(2s)^8 \zeta(6s - 1) 
\sum_{D} \frac{ h_3(D) + 1}{2} |D|^{-s}.
\]
In particular, the discriminant of any non-degenerate cubic ring may be expressed in the form $a^2b^6D$, and the number of non-isomorphic rings with discriminant factored in this way is at most $C\tau_8(a) b \cdot {(h_3(D)+1)}/2 \ll a^\epsilon b h_3(D)$ for any $\epsilon>0$.

Noting that $\tau(a^2b^6D) \ll a^\epsilon b^\epsilon \tau(D)$, we therefore obtain by Lemma \ref{eq:divisor_bound} that
\begin{align*}
\sum_{0 < |\Disc(x)| < Y} \tau(|\Disc(x)|) 
	& \ll \sum_a \sum_b a^\epsilon b^{1+\epsilon} \sum_{0 < |D| < \frac{Y}{a^2 b^6}} \tau(D)h_3(D) \\
	& \ll Y \log Y \sum_a \sum_b a^{-2+\epsilon} b^{-5+\epsilon} \\
	& \ll Y \log Y.
\end{align*}
In fact, though the implied constants above depend on $\epsilon$, we obtain the conclusion with an absolute implied constant on choosing $\epsilon = 1/10$, say.

To prove \eqref{eq:divisor_bound_2}, we may assume that $q$ is odd by absorbing any $2$-adic condition into the implied constant.  Since we have written $\Disc(x)$ in the form $a^2b^6D$, the condition that $q^2 \mid \Disc(x)$ implies that $q \mid ab$.  Proceeding as above, we find
\begin{align*}
\sum_{\substack{|\Disc(x)| < Y \\ q^2 \mid \Disc(x)}} \tau(|\Disc(x)|)
	\ll Y \log Y \sum_{\substack{ a,b \\ q \mid ab}} a^{-2+\epsilon} b^{-5+\epsilon}
	\ll \frac{1}{q^{2-\epsilon}} Y\log Y
\end{align*}
as claimed.
\end{proof}

We come now to the key formula (60) of \cite{BTT}. We have, for any $A, Q > 0$, that 

\begin{equation}\label{eq:lc_setup}
\sum_{d < X^{\frac{1}{2}} (\log X)^{-A}} 
\sum_{\substack{\Sigma_m \\ m < (\log X)^B}} 
\left|
E'(X_{d, m}, d, \Sigma_m)
 \right|
\ll E_1 + E'_1 + E_2 + E_3,
\end{equation}
with 
\begin{align*}
E_1 := & \ X \sum_{d < X^{\frac12} (\log X)^{-A}} 
 \sum_{\substack{q > Q}} \sum_{m,\Sigma_m} 
\left|\Res_{s = 1} \xi(s, \Phi_{d, q, m})\right|,
\\
E'_1 := & \ X^{5/6} \sum_{d < X^{\frac12} (\log X)^{-A}} 
 \sum_{q \geq 1} \sum_{m,\Sigma_m} 
\left|\Res_{s = 5/6} \xi(s, \Phi_{d, q, m})\right|,
\\
E_2 := & \ X^{\frac35}
\sum_{D, Q_1, M} 
\Bigg( \sum_{q, d, m, \Sigma_m} \hspace{-0.1cm} \Res_{s = 1} \xi(s, \Phi_{d, q, m})
\Bigg)^{\frac35} 
\Bigg( 
(D Q_1^2 M^2)^4 \sup_{Y} \frac{1}{Y}
\sum_{q, d, m, \Sigma_m} N(Y, |\widehat{\Phi_{d, q, m}}|) \Bigg)^{\frac25}
\\
= & \ X^{\frac35}
\sum_{D, Q_1, M} E_{2}'(D, Q_1, M)^{3/5} {\widehat E}_{2}'(D, Q_1, M)^{2/5},
\\
E_3 := &
\sum_{d < X^{\frac12} (\log X)^{-A}} \sum_{q > Q} \sum_{m,\Sigma_m} N(X, \Phi_{d, q, m}).
\end{align*}
For convenience, here we write $\Res_{s = 1} \xi$ as shorthand for $\Res_{s = 1} (\xi^+ + \xi^-)$.
In the sum defining $E_2$, the outer sums over $D$, $Q_1$, and $M$ are over integer powers of $2$ less than $X^{1/2} (\log X)^{-A}$, $Q$, and $(\log X)^B$ respectively, and the inner sums
over $q, d, \Sigma_m$ are over $q \in [Q_1, 2Q_1]$, $d \in [D, 2D]$, $m \in [M, 2M]$, and all local conditions $\pmod{m}$. 

\medskip

{\itshape The residue tail sums: Bounds on $E_1$ and $E'_1$.} By the explicit residue formulas in \cite[Theorem 2.4]{BTT}, together with Gronwall's classical bound
$\sigma(n) \ll n \log \log n$ \cite{gronwall}, we have
\begin{align}
 \Res_{s = 1} \xi(s, \Phi_{d, q, m}) & \ll d^{-1} q^{-2 + \epsilon} (q, d) \prod_{p \mid d} (1 + p^{-1}) \nonumber \\
& \ll d^{-1} \log \log d \cdot (q, d) q^{-2 + \epsilon} \label{eq:residue_bound}.
\end{align}
Noting that the number of local conditions $\Sigma_m \pmod{m}$ is $O_{\epsilon}(m^\epsilon)$, by summing over $d < X^{1/2} (\log X)^{-A}$, $q > Q$, and $m < (\log X)^B$, we see that 
\begin{align*}
E_1 \ll & \ X (\log X)^{B + \epsilon} \sum_{q > Q} q^{-2 + \epsilon} \sum_{r \mid q} r \sum_{\substack{d < X^{1/2} (\log X)^{-A} \\ r \mid d}} d^{-1} \log \log d \\
\ll & \ X (\log X)^{B + \epsilon} \sum_{q > Q} q^{-2 + \epsilon} \tau(q) (\log X)^{1 + \epsilon} \\
\ll & \ X (\log X)^{B + 1+ \epsilon} \cdot Q^{-1 + \epsilon}.
\end{align*}
In $E'_1$, in contrast to \cite{BTT} the sum is over all squarefree integers $q$, as the secondary terms in the Davenport--Heilbronn theorems are being treated as error terms. We have
\[
 \Res_{s = 5/6} \xi(s, \Phi_{d, q, m}) \ll d^{-1 + \epsilon} q^{-5/3 + \epsilon}
 \]
by \cite[Theorem 2.4]{BTT}, so that $E'_1 \ll X^{5/6 + \epsilon}$.

{\itshape The counting tail sum: Bounds on $E_3$.} Using \eqref{eq:divisor_bound_2} from Lemma~\ref{lem:divisor_bound_1} gives
\begin{align}
\sum_{d < X^{\frac12} (\log X)^{-A}} \sum_{q > Q} \sum_{\Sigma_m} N(X, \Phi_{d, q, m}) 
& \ll (\log X)^{B+ \epsilon} \sum_{d < X^{\frac12} (\log X)^{-A}} \sum_{q > Q} N(X, \Phi_{d, q}) \nonumber \\
& \ll X (\log X)^{B + 1 + \epsilon} \cdot Q^{-1 + \epsilon}. \nonumber
\end{align}

{\itshape The Landau's method error sum: Bounds on $E_2$.}
The sum $E_2$ arises by applying Landau's method to the sum over $q \leq Q$. By \eqref{eq:residue_bound} we have
\[
E'_2(D, Q_1, M) \ll M^{1 + \epsilon} \cdot Q_1^{-1 + \epsilon} \log \log D.
\]
We come now to $\widehat{E_2'}(D, Q_1, M)$, which is the technical heart of the matter, and our point of departure
from \cite{BTT}. We begin by assembling facts 
about the Fourier transform $\widehat{\Phi_{d, q, m}}$. It is multiplicative in the modulus, and we decompose
\[
\widehat{\Phi_{d, q, m}} = \prod_{p \mid d, p \nmid qm} \widehat{\Phi_p} 
\prod_{p \mid qm} \widehat{\Psi_p},
\] 
where $\widehat{\Psi_p}$ stands for any function for which we will apply only the trivial bound
$|\widehat{\Psi_{p}}(x)| \leq 1.$ The $q$-aspect required great care in \cite{BTT}, but here the trivial bound will quite suffice.

The Fourier transform of $\Phi_d$ is the important one here, and we have the following explicit formula: 

\begin{proposition}\label{prop:mori}
Let $p \neq 3$ be a prime.
The function
$\Phi_p\colon V(\Z/p\Z) \rightarrow \{ 0, 1 \}$
satisfies 
$$
|\widehat{\Phi_p}(x)| =
\begin{cases}
	p^{-1} + p^{-2} - p^{-3}	& x=0,\\
	p^{-2} - p^{-3}	& x \neq 0, \ \ p \mid \Disc(x),\\
		- p^{-3}		& \textup{otherwise}.
\end{cases}
$$
\end{proposition}

This was stated in weaker form in \cite[Proposition 6.4]{BTT}. The exact formula above is due essentially
to Mori \cite{mori}, and was obtained in this form by Taniguchi--Thorne \cite[Proposition 6.1]{TT_L}.

We proceed by noting that
\[
\sup_Y \frac{1}{Y} \sum_{q \in [Q_1, 2Q_1]} \sum_{d \in [D, 2D]} \sum_{\substack{\Sigma_m \\ m \in [M, 2M]}} N(Y, |\widehat{\Phi_{d, q, m}}|)
\ll
Q_1 M^{1 + \epsilon} \sup_Y \frac{1}{Y} \sum_{d \in [D, 2D]} N(Y, |\widehat{\Phi_d}|).
\]
%
Write $S(D, Y)$ for the inner sum over $d$, so Proposition~\ref{prop:mori} gives
\[
S(D, Y) \leq \sum_{d_1 d_2 d_3 \in [D, 2D]} \frac{2^{\omega(d_1)}}{d_1 d_2^2 d_3^3} 
\sum_{\substack{|\Disc(x)| < Y/d_1^4 \\ d_2 \mid \Disc(x)}} 1,
\]
where we apply the trivial bound in place of Proposition \ref{prop:mori} for $p = 3$.
We dyadically decompose the summation over each $d_1,d_2,d_3$ to cover this sum with $O((\log D)^2)$ sums of the form
\begin{align}
& \sum_{d_1 \in [D_1, 2D_1]} \sum_{d_2 \in [D_2, 2D_2]} \sum_{d_3 \in [D_3, 2D_3]} \frac{2^{\omega(d_1)}}{d_1 d_2^2 d_3^3} 
\sum_{\substack{|\Disc(x)| < Y/d_1^4 \\ d_2 \mid \Disc(x)}} 1 \label{eq:covered1} \\
\leq \ & \nonumber
(D_2 D_3)^{-2} \sum_{d_1 \in [D_1, 2D_1]} \frac{2^{\omega(d_1)}}{d_1} 
\sum_{\substack{|\Disc(x)| < Y/d_1^4}} \tau(|\Disc(x)|) \\
\ll \ & \nonumber
(D_2 D_3)^{-2} \sum_{d_1 \in [D_1, 2D_1]} \frac{2^{\omega(d_1)}}{d_1} 
\frac{Y}{d_1^4} (\log Y) \\
\ll \ & \nonumber
(D_2 D_3)^{-2} D_1^{-4 + \epsilon} Y \log Y,
\end{align}
where $D_1D_2D_3 \asymp D$.  Since the contribution from $D_1$ is smaller than that of $D_2$ and $D_3$, this implies
$$
	S(D, Y) \ll D^{-2} Y (\log D) (\log Y).
$$
If $Y \leq D^{100}$, say, this is $O(D^{-2} Y (\log D)^2)$, which is sufficient for our purposes.  If $Y > D^{100}$, we obtain the same bound by instead applying standard contour integration techniques (e.g. Landau's method) to estimate the inner-most sum in \eqref{eq:covered1}. As $D \leq X$, we obtain
\begin{equation}
\label{eq:hat_e2_bound}
\widehat{E_2'}(D, Q_1, M) \ll D^2 Q_1^9 M^{9 + \epsilon} (\log X)^2.
\end{equation}
In conclusion, we have that $(E_1 + E_3) + E_1' + E_2$ is at most
\begin{align*}
\ll & \frac{X (\log X)^{B + 1+ \epsilon}}{Q^{1 - \epsilon}} +X^{\frac56 + \epsilon}
+ X^{\frac35} \sum_{D, Q_1, M} (M^{1 + \epsilon} Q_1^{-1 + \epsilon} \log \log D)^{\frac35} (D^2 Q_1^9 M^{9 + \epsilon} (\log X)^2)^{\frac25} \\
\ll & \frac{X (\log X)^{B + 1+ \epsilon}}{Q^{1 - \epsilon}}
+ X (\log X)^{-\frac{4}{5} A + \frac{4}{5} + \frac{21}{5} B + \epsilon} Q^{3 + \epsilon},
\end{align*}
where $D$, $Q_1$, and $M$ are summed over powers of $2$ less than $X^{1/2} (\log X)^{-A}$, $Q$, and $(\log X)^B$ respectively. Choosing $Q = (\log X)^{1.01 + B + k}$ and $A = 5 + 9B + 5k$ we thus obtain
\[
\sum_{\substack{d < X^{1/2} (\log X)^{-A} \\ d \textnormal{ squarefree}}} |E(X_d, d)| \ll X (\log X)^{-k},
\]
as desired.

\medskip
{\itshape Two technical remarks:} 

\ (1) The formula \eqref{eq:lc_setup} was proved in \cite{BTT} subject to the technical condition that for each $D$, $Q_1$, and $M$ we have
${\widehat E'}_2(D, Q_1, M) \ll E'_2(D, Q_1, M) X.$ This can be checked by \cite[Theorem 2.4]{BTT}, which implies that 
$E'_2(D, Q_1, M) \gg M Q_1^{-1}$, and comparing to \eqref{eq:hat_e2_bound}.

\ (2) The above corrects a subtle mistake in \cite{BTT}, where in the definition of $\widehat{E'_2}$ the supremum over $Y$ and the sum over local conditions were given in the opposite order. In fact, the stronger bound stated above holds, as indeed was applied in \cite{BTT}. This is subject to the requirement that the parameters $q, d, m, X_{d, m}$ each vary only within a dyadic interval, or more generally within intervals of the form $[Z, CZ]$ for a fixed constant $C > 1$.

\section{The average 6-torsion}
In this section, we prove our main results (Theorems \ref{thm:main-} and \ref{thm:main+}) for the $6$-torsion in (narrow) class groups of quadratic fields. 

\label{s6torsion}
\subsection{Bounding divisors around the square root barrier}
\label{ss6Trick}
We begin by bridging the gap between the level of distribution of $X^{1/2} (\log X)^{-5}$ proved in \S \ref{sLevel}, and the `ideal' level of distribution
of $X^{1/2}$ as described in the introduction.

This Hooley Delta function was introduced by Erd\H os~\cite{erdos}; Hooley~\cite{hooley} proved that its average is $O((\log n)^\theta)$ for some $\theta<1$. This was later improved significantly by Hall--Tenenbaum~\cites{hall1,hall2,hall3} and La Bret{\`e}che--Tenenbaum~\cite{brete},
whose work allows to take arbitrary $\theta>0$ (this is an oversimplification of their work, that is in fact much stronger). Koukoulopoulos and Tao~\cite{koukoutao} recently made a breakthrough by showing that the average is $O((\log \log n)^{11/4})$. Finally, 
work by La Bret{\`e}che--Tenenbaum~\cite{brete2} improved this into 
\begin{equation}
\label{hooleydeltaestimate}
\sum_{n\leq x} \Delta(n) \ll x (\log \log x)^{5/2}.
\end{equation}

Let $h_3(D)$ denote the size of the $3$-torsion of $\mathbb Q(\sqrt D)$.

\begin{lemma}
\label{hooleylemma} 
We have for all $L \geq 1$ and for all $T \geq 10$
$$
\sum_{0<-D\leq T} (h_3(D)-1) \#\{d\in \mathbb N:d\mid D, \sqrt T(\log T)^{-L}\leq d \leq \sqrt{T}(\log T)^{L}\} \ll L T (\log \log T)^{7/2},
$$ 
where the sum is over negative squarefree discriminants $D$ and the implied constant is absolute. The same estimate holds for positive discriminants.
\end{lemma}

\begin{proof} 
Our argument will be rather similar to \cite[Lemma 3.1]{4authorapplications}, except that we replace the indicator function of the interval $\mathbf{1}_{|D| \in [0, T]}$ by $\Delta(|D|)$. To keep our notation consistent with the argument in \cite[Lemma 3.1]{4authorapplications}, we shall use the level of distribution result for $h_3(D)$ from ~\cite[\S 3.1]{4authorapplications}.
We split $[\sqrt T(\log T)^{-L}, \sqrt{T}(\log T)^{L}]$ into intervals of the form 
$$
[\sqrt T(\log T)^{-L}\mathrm e^i,\sqrt T(\log T)^{-L}\mathrm e^{i+1})
$$ 
for $i = 0, 1, \ldots$. The number of intervals we need to cover the whole range is $O(L \log \log T)$ with an absolute implied constant. Hence, the sum in the lemma is 
$$
\ll L (\log \log T) \sum_{0 < -D \leq T} (h_3(D)-1) \Delta(|D|).
$$ 
Our goal will be to employ~\cite[Theorem 1.9]{4author} with $f = \Delta$. This result takes the following shape: given an infinite set $\mathcal{A}$, given a function $\chi_T: \mathcal{A} \rightarrow [0, \infty)$ of finite support for each $T \geq 1$ and a sequence of strictly positive integers $\mathfrak{C} = (c_a)_{a \in \mathcal{A}}$, then we obtain, under suitable assumptions, a sharp upper bound for
$$
\sum_{a \in \mathcal{A}} \chi_T(a) f(c_a).
$$
Now take
$$
\mathcal A = \{D \textrm{ negative discriminant}\}, \ \ c_D = -D, \ \ \chi_T(D) = (h_3(D)-1) \mathbf{1}_{0 < -D \leq T}, \ \ M(T) = \frac{4T}{\pi^2},
$$ 
so that 
$$
\sum_{D \in \mathcal{A}} \chi_T(D) \Delta(c_D) = \sum_{0 < -D \leq T} (h_3(D)-1) \Delta(|D|).
$$ 
The suitable assumptions in \cite[Theorem 1.9]{4author} are concretely:
\begin{itemize}
\item the function $f = \Delta$ is in the class $\mathcal{M}(A, \epsilon, C)$ from \cite[Definition 1.2]{4author}. Indeed, we have $\Delta(mn)\leq \Delta(m)\tau(n)$ for coprime $m, n$, where $\tau$ is the divisor function,
\item the sequence $\mathfrak{C}$ has a level of distribution, see \cite[Definition 1.6]{4author},
\item the sequence $c_a$ has at most polynomial growth, as in \cite[Equation (1.8)]{4author}, which is clearly satisfied (because $c_a = |a| \leq T \ll M(T)$ if $\chi_T(a) > 0$),
\item the total mass satisfies
$$
\lim_{T \rightarrow +\infty} \sum_{a \in \mathcal{A}} \chi_T(a) = +\infty,
$$
which is again clearly satisfied in our situation.
\end{itemize}
So it remains to check that $\mathfrak{C}$ has a level of distribution. Let $g_1$ be the multiplicative function defined through 
\[
g_1(p^e) =
\begin{cases}
p/(p+1), & \textup{if } p \geq 2 \textup{ and } e = 1 \\
0, & \textup{if } p > 2 \textup{ and } e \geq 2 \\
4/3, & \textup{if } p = 2 \textup{ and } e = 2 \\
4/3, & \textup{if } p = 2 \textup{ and } e = 3 \\
0, & \textup{if } p = 2 \textup{ and } e \geq 4,
\end{cases}
\] 
and set $h_T(q) = g_1(q)/q$. Now the level of distribution assumption is verified in~\cite[\S 3.1]{4authorapplications} where the density function is $h_T(q)$ as defined above and the main term is $M = M(T) = 4T/\pi^2$. Note that in the notation of~\cite[Theorem 1.9]{4author}, the function $M(T)$ is abbreviated as $M$ and that in this particular application the function $h_T(q)$ is independent of $T$. Hence,~\cite[Theorem 1.9]{4author} shows the upper bound
$$
\sum_{D \in \mathcal{A}} \chi_T(D) \Delta(c_D) \ll T \prod_{\substack{p \leq M}} (1 - h_T(p)) \sum_{\substack{a \in \mathbb N \cap [1,M]}} \Delta(a) h_T(a).
$$
To conclude the proof note that 
$$
\prod_{p \leq M} (1-h_T(p)) \ll \prod_{2<p\leq M} \left(1-\frac{1}{p+1}\right) \ll \frac{1}{\log M}\ll \frac{1}{\log T}.
$$
Finally, $g_1$ is bounded by definition, thus, 
$$
\sum_{\substack{ a\in \mathbb N \cap [1,M] }} \Delta(a) h_T(a) 
\ll\sum_{\substack{ a\in \mathbb N \cap [1,M] }} \frac{\Delta(a) }{a} \ll
(\log T) (\log \log T)^{5/2}
$$ 
by~\eqref{hooleydeltaestimate} and partial summation. Putting these estimates together concludes the proof.
\end{proof}

%


\subsection{Obtaining the asymptotic: proofs of Theorems \ref{thm:main-} and \ref{thm:main+}}
\label{ss6Main}
To prove Theorem \ref{thm:main-}, as a first step we split
\begin{align*}
\sum_{\substack{-X < D < 0 \\ D \text{ fundamental}}} h_6(D) &= \sum_{\substack{-X < D < 0 \\ D \text{ fundamental}}} h_2(D) h_3(D) \\
&= \sum_{\substack{-X < D < 0 \\ D \text{ fundamental}}} h_2(D) + \sum_{\substack{-X < D < 0 \\ D \text{ fundamental}}} h_2(D) (h_3(D) - 1).
\end{align*}
We have
\begin{align}\label{eq:2tors_est}
\sum_{\substack{-X < D < 0 \\ D \text{ fundamental}}} h_2(D) & = \frac{1}{2} \sum_{\substack{-X < D < 0 \\ D \text{ fundamental}}} 2^{\omega(D)} 
= \sum_{q < \sqrt{X}} \mu^2(q) \sum_{\substack{-X < D < 0 \\ q^2 < \textnormal{sfk}(D) \\ q \mid D}} 1,
\end{align}
where $\textnormal{sfk}(D)$ denotes the largest positive squarefree divisor of $D$.

We estimate the inner sum by \cite[Proposition 4.2]{BTT}, applying it twice to handle both the upper bound on $-D$ in \eqref{eq:2tors_est} as well as the lower bound on $\textnormal{sfk}(D)$.
For the latter, note that $D/\textnormal{sfk}(D)$ depends only on the $2$-part of $D$, and we use \cite[Proposition 4.2]{BTT} to control this $2$-part as well as the condition $q \mid D$. 

We obtain that
\begin{align}
\sum_{\substack{-X < D < 0 \\ D \text{ fundamental}}} h_2(D) & = 
 \sum_{q < \sqrt{X}} \mu^2(q) 
 \bigg( \frac{3}{\pi^2} \frac{g^*(q)}{q} X + O(X^{1/2} q^{-1/2 + \epsilon}) + O(q) \bigg) \label{eq:pie1} \\
 & = 
\frac{3}{\pi^2} X \sum_{q < \sqrt{X}} \mu^2(q) 
\frac{g^*(q)}{q} + O(X), \label{eq:pie2}
 \end{align}
where $g^*(q)$ was defined in \eqref{def:omega1}. Note that the error term of $O(q)$ in \eqref{eq:pie1} comes from
the main asymptotic term arising when applying \cite[Proposition 4.2]{BTT} to the lower bound on $\textnormal{sfk}(D)$.

To handle the second sum, in the innermost sum in \eqref{eq:2tors_est} we replace $1$ 
by $h_3(D) - 1$, which equals twice the number of cubic fields of discriminant $D$ up to isomorphism.
We apply the level of distribution estimate of Theorem \ref{thm:ld} to the sum over squarefree $q < \frac{1}{4} \sqrt{X} (\log X)^{-5}$, and the 
Hooley delta estimate of Lemma \ref{hooleylemma} to the sum over $\frac{1}{4} \sqrt{X} (\log X)^{-5} \leq q < \sqrt{X}$, to conclude that
\begin{equation}\label{eq:applied_lod}
\sum_{\substack{-X < D < 0 \\ D \text{ fundamental}}} h_2(D) (h_3(D) - 1) = 
2 \cdot \frac{3}{2 \pi^2} X \sum_{q < \frac{1}{4} \sqrt{X} (\log X)^{-5}} \mu^2(q) \frac{g^*(q)}{q} 
+ O(X (\log \log X)^{7/2}).
\end{equation}

Recall that we have
\[
\frac{g^*(q)}{q} = \prod_{p \mid q} \frac{1}{p + 1}.
\]
Therefore, applying \cite[Theorem A.5]{FI} to \eqref{eq:pie2} and \eqref{eq:applied_lod}, we obtain that
\begin{equation}
\label{eq:main_concl}
\sum_{\substack{-X < D < 0 \\ D \text{ fundamental}}} h_6(D) = 
\frac{3}{\pi^2} X \log X \prod_p \left(1 + \frac{1}{p + 1}\right) \left(1 - \frac{1}{p}\right) 
+ O(X (\log \log X)^{7/2}).
\end{equation}
This yields Theorem \ref{thm:main-}. 

The proof of the first part of Theorem \ref{thm:main+} is identical, except that the estimate in \eqref{eq:applied_lod} is multiplied by $\frac{1}{3}$. 

Finally, to prove the second part of Theorem \ref{thm:main+}, recall the formula \eqref{eh2D}. For any $\epsilon > 0$ choose $Z := Z(\epsilon)$ such that
\begin{equation}\label{eq:ram_prod}
\prod_{\substack{p \leq Z \\ p \equiv 3 \pmod{4}}} \bigg( \frac{p}{p + 1} \bigg) < \epsilon.
\end{equation}
We repeat the above proof, imposing the local specification that $p \nmid D$ for each $p \leq Z$ with $p \equiv 3 \pmod{4}$,
and reducing the level of distribution to $\sqrt{X} (\log X)^{-6}$ so as to allow these local specifications. 

We thus obtain the analogue of \eqref{eq:main_concl} for such $D$, with the main term multiplied by the product in \eqref{eq:ram_prod}.
As this is negligible as $\epsilon \rightarrow 0$, we obtain the result.

\subsection{A technical generalization involving local conditions}\label{sslocal}
We have now obtained asymptotic formulas for $\sum_{|\Disc(K)| < X} \tau(\Disc(K))$, where the sum is
over cubic fields of fundamental discriminant. We now prove a technical generalization of these results which will be applied in \S \ref{sec:malle}.

We first introduce some notation. Given a squarefree integer $\beta$, we define $\beta'$ to be its largest positive divisor coprime to $6$. We also fix local behaviors at $2$, $3$, and $\infty$. More precisely, we write $\Sigma$ to indicate choices of a cubic \'etale extension of $\R$, of $\Q_2$, and of $\Q_3$, and say that $\phi: G_\Q \twoheadrightarrow S_3$ satisfies $\Sigma$ if the cubic subfield of $F$ (unique up to isomorphism) has the designated respective local completions. 

For a positive integer $q$ coprime to $6$, we write $\mathfrak{m}_\beta(q, \Sigma)$ for the number of epimorphisms $\phi: G_\Q \twoheadrightarrow S_3$ with quadratic subfield $\Q(\sqrt{\beta})$, completions as specified by $\Sigma$, and such that $q$ is the product of primes $\neq 2, 3$ which are totally ramified in the cubic subfield. If $q$ and $\beta$ are not coprime, then we have $\mathfrak{m}_\beta(q, \Sigma) = 0$.

We now calculate 
$$
\sum_{\substack{\beta \text{ squarefree} \\ \beta' \leq t, \ e \mid \beta'}} \mathfrak{m}_\beta(q, \Sigma) \tau(\beta')
$$
with some uniformity in $t$ with respect to $e$ and $q$, where the sum is over squarefree integers $\beta$ satisfying $\beta' \leq t$ and $e \mid \beta'$. 

This sum vanishes unless $\gcd(e, 6q) = 1$ and $\mu^2(e) = 1$, which we will henceforth assume. Note that the local conditions in $\Sigma$ certainly determine the sign, the $2$-adic and the $3$-adic valuation of $\beta$. Take $\delta \in \{-6, -3, -2, -1, 1, 2, 3, 6\}$ to correspond to these conditions, so that
$$
\sum_{\substack{\beta' \leq t \\ e \mid \beta'}} \mathfrak{m}_\beta(q, \Sigma) \tau(\beta') = \sum_{\substack{n \leq t \\ e \mid n}} \mathfrak{m}_{\delta n}(q, \Sigma) \tau(n) \mu^2(6 n) = \tau(e) \sum_{m \leq t/e} \mathfrak{m}_{\delta e m}(q, \Sigma) \tau(m) \mu^2(6 m e).
$$
We rewrite the sum as
$$
\sum_{m \leq t/e} \mathfrak{m}_{\delta e m}(q, \Sigma) \tau(m) \mu^2(6 m e) = 2\sum_{d \leq (t/e)^{1/2}} \sum_{\substack{d^2 \leq m \leq t/e \\ d \mid m}} \mathfrak{m}_{\delta e m}(q, \Sigma) \mu^2(6 m e) + O(1),
$$
where the $O(1)$ term accounts for the case $m = 1$. In this summation, we see that $d$ must be coprime to $6qe$ as the sum vanishes otherwise. Fixing such a $d$, note that the elements in $\mathfrak{m}_{\delta e m}(q, \Sigma)$ occurring in the inner sum correspond to six times the number of cubic fields $F$ (up to isomorphism) satisfying the conditions:
\begin{itemize}
\item the cubic \'etale algebras $F \otimes \Q_2$, $F \otimes \Q_3$ and $F \otimes \R$ are given by $\Sigma$, 
\item the discriminant of $F$ is exactly divisible by $de$,
\item the quantity $\text{Disc}(F)/q^2$ is not divisible by $p^2$ for any prime $p \geq 5$,
\item we have 
$$
d^2 \leq \frac{\text{Disc}(F)}{\text{Disc}(\Sigma_2) \text{Disc}(\Sigma_3) e q^2} \leq t/e.
$$
\end{itemize}
The second condition can be rephrased as $F \otimes \Q_p$ being a partially ramified \'etale algebra for all $p \mid de$, while the third condition breaks into two distinct conditions: namely $F \otimes \Q_p$ is a totally ramified \'etale algebra for all $p \mid q$, and moreover $F \otimes \Q_p$ is not a totally ramified \'etale algebra for all $p \nmid q$ with $p \geq 5$. Applying the level of distribution result gives a main term of
$$
12 \sum_{d \leq (t/e)^{1/2}} \mu^2(6qde) \frac{w_\infty w_2 w_3 t q^2}{3 \zeta(3)} \prod_{p \mid de} \left(\frac{p}{p^2 + p + 1}\right) \prod_{p \mid q} \left(\frac{1}{p^2 + p + 1}\right) \prod_{p \nmid 6deq} \left(\frac{p^2 + p}{p^2 + p + 1}\right),
$$
where $w_p$ is given by
$$
w_p := \frac{1}{|\Aut(\Sigma_p)| m_p}, \quad \quad m_p := \sum_{E \in A_p} \frac{1}{\text{Disc}_p(E) |\Aut(E)|}
$$
and where $A_p$ denotes the set of cubic \'etale algebras of $\Q_p$ up to isomorphism. This equals
$$
\frac{4 w_\infty w_2 w_3 t q^2}{\zeta(3)} \prod_{p \mid e} \left(\frac{1}{p + 1}\right) \prod_{p \mid q} \left(\frac{1}{p^2 + p}\right) \prod_{p \nmid 6} \left(\frac{p^2 + p}{p^2 + p + 1}\right) \sum_{d \leq (t/e)^{1/2}} \mu^2(6qde) \prod_{p \mid d} \left(\frac{1}{p + 1}\right),
$$
which is
$$
\frac{4 \cdot 7 \cdot 13 \cdot w_\infty w_2 w_3 t q^2}{6 \cdot 12 \cdot \zeta(2)} \prod_{p \mid e} \left(\frac{1}{p + 1}\right) \prod_{p \mid q} \left(\frac{1}{p^2 + p}\right) \sum_{d \leq (t/e)^{1/2}} \mu^2(6qde) \prod_{p \mid d} \left(\frac{1}{p + 1}\right).
$$
Finally, we evaluate the inner sum 
$$
\frac{1}{2} \log(t/e) \prod_p \left(1 - \frac{1}{p}\right) \left(1 + \frac{\mathbf{1}_{p \nmid 6qe}}{p + 1}\right)
$$
thanks to \cite[Theorem A.5]{FI}. As for the error term, it is
\[
\sum_{d \leq (t/e)^{1/2}} \left[ E^*( D tq^2, de, \Sigma_{6q})
+ E^*(D d^2eq^2, de, \Sigma_{6q}) \right]
\]
with $D := \text{Disc}(\Sigma_2) \text{Disc}(\Sigma_3)$ and $\Sigma_{6q}$, denoting the local conditions given by $\Sigma$ and that primes dividing $q$
totally ramify. The error term may be bounded by Theorem \ref{thm:ld}, with $X := C t (\log t)^{2B_1}$ for a sufficiently large constant $C$, in the range $d \leq (t/e)^{1/2} (\log t)^{-5 - 18B_1 - B_2/2}$, and the remaining range of $d$ is handled by Lemmas \ref{hooleylemma} and \ref{lem:easy-bound}. 

Wrapping up, we have proven

\begin{theorem}
\label{tmAverage}
Let $B_1, B_2 > 0$ be fixed. Then we have for all $t \geq 100$, 
for all positive squarefree integers $q < (\log t)^{B_1}$ and $e < (\log t)^{B_2}$ with $\gcd(q, 6) = \gcd(e, 6) = \gcd(q, e) = 1$,
and for all choices $\Sigma$ of a cubic \'etale extension of $\R$, of $\Q_2$, and of $\Q_3$
\begin{multline*}
\sum_{\substack{\beta' \leq t \\ e \mid \beta'}} \mathfrak{m}_\beta(q, \Sigma) \tau(\beta') = \frac{91 c w_\infty w_2 w_3 \tau(e) t \log(t/e)}{36 \cdot \zeta(2)} \prod_{p \mid e} \left(\frac{1}{p + 1}\right) \prod_{p \mid q} \left(\frac{p}{p + 1}\right) \\
+ O_{B_1, B_2}(\tau_4(q) t (\log \log t)^{7/2})
\end{multline*}
with
$$
c := \prod_p \left(1 - \frac{1}{p}\right) \left(1 + \frac{\mathbf{1}_{p \nmid 6qe}}{p + 1}\right).
$$
\end{theorem}

\section{\texorpdfstring{Malle's conjecture for Galois $D_6$-extensions}{Malle's conjecture for Galois D6-extensions}}
\label{sec:malle}
The goal of this section is to count regular $D_6 \simeq C_2 \times S_3$-extensions of $\mathbb{Q}$. We start by establishing the tail bounds that will reduce the count of $D_6$-extensions to the sum of $h_6(d)$ over $d$ satisfying finitely many local conditions. We will then evaluate the resulting sum using the results from \S \ref{s6torsion}. For now we record the estimate \cite[Theorem 1.3]{4authorapplications}
	\begin{equation} \label{eqn:ckps-bound}
		\sum_{d \leq Y} \tau_s(d) h_3(d)
			\ll_s Y (\log Y)^{s - 1},
	\end{equation}
which we shall use throughout this section.

\subsection{Initial observations}
Let $G = D_6 \simeq C_2 \times S_3$ be the dihedral group of order $12$. This group has two non-isomorphic faithful transitive permutation representations, one in degree $6$ (already counted in the literature) and one in degree $12$. The degree $12$ representation of $G$ is just the regular representation (corresponding to Galois extensions), and is essentially equivalent to the $6$-torsion problem for reasons to be seen shortly.

For any $X \geq 1$, let
	\[
		\mathcal{F}(X;G)
			:= \{ K/\mathbb{Q} \text{ Galois} : |\mathrm{Disc}(K)| \leq X, \mathrm{Gal}(K/\mathbb{Q}) \simeq G\}.
	\]
Note that the regular representation is nothing more than $G$ acting on the set $X := G$ by left multiplication. Given $g \in G$, it follows that all the orbits $\{gx : x \in X\}$ have length exactly $\ord(g)$, so the number of orbits is $|G|/\ord(g)$. Hence we have
$$
\mathrm{ind}(g) = |G| - \frac{|G|}{\ord(g)}.
$$
In particular, if we write $p$ for the smallest prime divisor of $|G|$, it follows that
$$
a(G) = |G| - \frac{|G|}{p}.
$$
We conclude that $a(D_6) = 6$. Moreover, we observe that $G$ has $5$ non-identity conjugacy classes. There are three conjugacy classes of minimal index (which are just the elements of order $2$), so Malle's conjecture states that 
	\[
		\#\mathcal{F}(X;G) \sim c X^{1/6} (\log X)^2
	\]
for some $c > 0$. We have tabulated the different conjugacy classes in Table \ref{tConjClasses}. 

\begin{table}[!ht]
\begin{center}
	\begin{tabular}{c|l|l} 
		Class & Order & Malle index \\ 
		$1 \times (2)$ & $2$ & $6$ \\
		$(2) \times (2)$ & $2$ & $6$ \\
		$(2) \times 1$ & $2$ & $6$ \\
		$1 \times (3)$ & $3$ & $8$ \\
		$(2) \times (3)$ & $6$ & $10$ 
	\end{tabular}
	\begin{caption}{
		\label{tConjClasses}
		Non-identity conjugacy classes in $C_2 \times S_3$.}
	\end{caption}
\end{center}
\end{table}

For ultimately getting the asymptotic, we introduce $G_\Q := \Gal(\overline{\Q}/\Q)$ and
	\[
		\mathrm{Epi}(X;G)
			:= \{\phi: G_\Q \twoheadrightarrow G : |\mathrm{Disc}(\phi)| \leq X\}.
	\]
Since the automorphism group of $D_6$ has size $12$, we observe that
$$
\# \mathcal{F}(X; G) = \frac{\#\mathrm{Epi}(X; G)}{12}.
$$
Viewing $G$ as $C_2 \times S_3$, we may write any homomorphism $\phi: G_\Q \rightarrow G$ as a pair $(\phi_1, \phi_2)$. Since $\phi_1$ is simply a quadratic character, we write $\phi_1 = \chi_\alpha$ for a unique squarefree integer $\alpha$ that we allow to be negative. We denote by $\chi_\beta$ the quadratic character coming from $\phi_2: G_\Q \twoheadrightarrow S_3 \twoheadrightarrow C_2$ and we let $F$ be the corresponding Galois $S_3$-extension. Then $F/\Q(\sqrt{\beta})$ is a cyclic cubic field extension. 

For the pair $(\phi_1, \phi_2)$ to be an epimorphism, it is necessary and sufficient that $\phi_1$ and $\phi_2$ themselves are epimorphisms and that $\alpha \neq \beta$. 

We also fix local behaviors at $2$, $3$, and $\infty$. We write $\Sigma$ to indicate choices of a quadratic and a cubic \'etale extension of $\R$, of $\Q_2$, and of $\Q_3$, and say that $\phi$ satisfies $\Sigma$ if $\Q(\sqrt{\alpha})$ and the cubic subfield of $F$ (unique up to isomorphism) have the designated respective local completions. To each $\Sigma$ is associated a positive integer $\Delta(\Sigma)$ giving the total $6$-adic contribution to the discriminant of any global $D_6$-extension satisfying $\Sigma$, see \cite[\S 2.1]{Bhargava}.

Write $\mathfrak{m}_\beta(q, \Sigma)$ for the number of epimorphisms $\phi_2: G_\Q \twoheadrightarrow S_3$ with quadratic subfield $\Q(\sqrt{\beta})$, completions as specified by $\Sigma$, and such that $q$ is the product of primes $\neq 2, 3$ which are totally ramified in the cubic subfield. If $q$ and $\beta$ are not coprime, then we have $\mathfrak{m}_\beta(q, \Sigma) = 0$.

Our count becomes 
\begin{align}
\label{eInitial}
\# \mathcal{F}(X; G) = \frac{1}{12} \sum_{\Sigma} \mathcal{F}(X; G, \Sigma) 
\end{align}
with
\begin{align}
\label{eFXGE}
\mathcal{F}(X; G, \Sigma) = \sum_{\substack{\alpha, \beta \ \Sigma-\textup{compatible} \\ \alpha \neq \beta, \alpha, \beta \neq 1}}
\sum_{\substack{q \\ \Delta(\Sigma) [\alpha', \beta']^6 q^8 (\alpha', q)^{-4} \leq X}}
\mathfrak{m}_\beta(q, \Sigma),
\end{align}
subject to the following notational conventions:
\begin{notation}\label{ournotation} We assume throughout that:
\begin{itemize} 
\item 
$\alpha$ and $\beta$ run over squarefree integers, positive or negative, with $\beta \neq 1$, subject to the $\Sigma$-compatibility condition
described above. We will shortly allow $\alpha = 1$ or $\alpha = \beta$
and bound the error in so doing.
\item $q$ ranges over positive squarefree integers coprime to $6\beta$. 
\item $\alpha'$ and $\beta'$ 
denote the largest positive divisors of $\alpha, \beta$ not containing any factors of $2$ or $3$. Given the $\Sigma$-compatibility condition, these quantities determine $\alpha$ and $\beta$ uniquely.
\end{itemize}
\end{notation}

As seen in the proof of Lemma \ref{lem:divisor_bound_1}, work of Datskovsky and Wright \cite[Lemmas 6.1 and 6.2]{DW3} gives the following bound for $\mathfrak{m}_\beta(q, \Sigma)$:

\begin{lemma}\label{lem:easy-bound}
Let $\beta$ be a squarefree integer, and let $q$ be a squarefree integer coprime to $6$. Then we have
		\[
			\mathfrak{m}_\beta(q, \Sigma) \ll h_3(\beta) \tau_4(q),
		\]
	where, for any integer $r \geq 1$, $\tau_r( \cdot)$ denotes the $r$-fold divisor function. The implied constant above is absolute.
\end{lemma}

We will later need the following elementary lemma on counting squarefree integers: 

\begin{lemma}
\label{lem:squarefree-sum}
Let $a$, $c$, $m$ be integers such that $\gcd(am, c) = 1$. For any $Y \geq 1$ we have
\[
\#\{1 \leq q \leq Y \textup{ squarefree}: \mathrm{gcd}(q,m) = 1, c \mid q - a\} = \frac{6Y\psi(cm)}{\pi^2 \phi(c)} + O(Y^{1/2} \tau(m)),
\]
where the implied constant is absolute and
$$
\psi(m) = \prod_{p \mid m} \left(\frac{p}{p+1}\right).
$$
\end{lemma}

\begin{proof}
By M\"obius inversion we write the sum over $q \leq Y$ as
$$
\sum_{\substack{q \leq Y \\ c \mid q - a}} 
\bigg(\sum_{\substack{d^2 \mid q \\ \gcd(d, m) = 1}} \mu(d)\bigg) \bigg(\sum_{e \mid \gcd(q, m)} \mu(e)\bigg) 
= \sum_{e \mid m} \mu(e) 
\sum_{\substack{d \leq Y^{1/2} 
\\ \gcd(d,m) = 1}} \mu(d) 
\sum_{\substack{q \leq Y, \ c \mid q - a \\ d^2e \mid q}} 1.
$$ 
The sum over $q$ is non-empty only when $\gcd(d, c) = 1$, so that we obtain the count 
$$ 
\sum_{e \mid m} \mu(e)\sum_{\substack{ d \leq Y^{1/2} 
\\ \gcd(d, cm) = 1}} \mu(d)
\left( \frac{Y}{cd^2e}+O(1)\right)= \frac{Y}{c}
\sum_{e \mid m}\frac{\mu(e)}{e}
\sum_{\substack{ d \leq Y^{1/2} \\ \gcd(d,cm) = 1}}
\frac{\mu(d)}{d^2}
+ O(Y^{1/2} \tau(m)).
$$ 
Completing the sum over $d$ gives a satisfactory error term, while the main term equals 
\[ 
\frac{Y}{c}\sum_{e \mid m}\frac{\mu(e)}{e} 
\prod_{p \nmid cm} \left(1-\frac{1}{p^2} \right)
=\frac{6Y}{c\pi^2 } 
\prod_{p\mid m} \left(1-\frac{1}{p} \right)
\prod_{p\mid cm} \left(1-\frac{1}{p^2} \right)^{-1}
=\frac{6Y\psi(cm)}{\pi^2 \phi(c)},
\] 
as desired.
\end{proof}

\subsection{Truncation lemmas} 
We will now prove a sequence of lemmas which show that the sum in \eqref{eFXGE} can be truncated in several ways with negligible error. Most significant is Lemma \ref{lem:bounded-q2-q3}, which shows that the bulk of the contribution comes from small $q$.

\begin{lemma}
\label{lem:extra_cont} 
Including the missing terms corresponding to $\alpha = \beta$ and $\alpha = 1$ in \eqref{eFXGE} introduces an error 
at most $O\big(X^{1/6} (\log \log X)^3\big)$.
\end{lemma}

\begin{proof} 
By Lemma \ref{lem:easy-bound}, the contributions from $\alpha = \beta$ and $\alpha = 1$ to \eqref{eFXGE} are each 
\begin{align*}
\ll & \sum_{\beta} \sum_{\substack{q \\ \beta^6 q^8 (\beta, q)^{-4} \ll X}} h_3(\beta) \tau_4(q) \\
\leq & \sum_{\beta} \sum_{d \mid \beta} \tau_4(d) \sum_{\substack{r \ll X^{1/8} d^{-1/2} |\beta|^{-3/4}}} 
h_3(\beta) \tau_4(r) \\
\ll & \sum_{\beta} h_3(\beta) \sum_{d \mid \beta} \tau_4(d) \frac{X^{1/8}}{d^{1/2} |\beta|^{3/4}} \bigg(\log\bigg(\frac{X^{1/8}}{|\beta|^{3/4}}\bigg)\bigg)^3,
\end{align*}
writing $q = dr$ in the second line with $d = (\beta, q)$. We now split the sum over $\beta$ into $|\beta| \leq X^{1/6} (\log X)^{-16}$ and $X^{1/6} (\log X)^{-16} < |\beta| \ll X^{1/6}$. In the smaller range, we bound the logarithmic term by $(\log X)^3$ and bound $\sum_{d \mid \beta} \tau_4(d)/d^{1/2} \ll \prod_{p \mid \beta} (1 + 4p^{-1/2}) \ll \tau(\beta)$, and obtain a bound of $O(X^{1/6})$ by \eqref{eqn:ckps-bound} and partial summation.

In the larger range of $\beta$, the logarithmic term is bounded by $(\log \log X)^3$. We split the sum over $d$ into $d \leq X^{1/100}$ and $d > X^{1/100}$. We have $\sum_{d \mid \beta, d > X^{1/100}} \tau_4(d)/d^{1/2} \ll X^{-1/200 + \epsilon}$, so that the large $d$ contribute
\[
\ll 
X^{1/8} (\log \log X)^3 \cdot X^{-1/200 + \epsilon} \cdot \sum_{\beta} \frac{h_3(\beta)}{|\beta|^{3/4}} \ll X^{1/6}.
\]
In the small $d$ range, our bound is
\[
\ll X^{1/8} (\log \log X)^3 \sum_{d \leq X^{1/100}} \frac{\tau_4(d)}{d^{1/2}} \sum_{d \mid \beta} \frac{h_3(\beta)}{|\beta|^{3/4}}.
\]
By Theorem \ref{thm:ld}, or alternatively by Belabas's earlier work \cites{belabas_sieve} as quoted in \cite[Theorem 3.1]{4author}, each inner sum is $\ll X^{1/24} / d$, yielding a total contribution $\ll X^{1/6} (\log \log X)^3$. This completes the proof.
\end{proof}

%

Our first aim is to use Lemma~\ref{lem:easy-bound} to prove the following `tail estimate':

\begin{lemma}\label{lem:bounded-q2-q3}
	Let $1 \leq Q \leq (\log X)^{2025}$. Then the number of $K \in \mathcal{F}(X;G)$ with $q > Q$ is 
		\[
			\ll \frac{X^{1/6} (\log X)^2 (\log \log X)^4}{Q^{1/3}}.
		\]
\end{lemma}

\begin{proof}
By Lemma~\ref{lem:easy-bound}, we find that the number of $K$ with $q > Q$ is
		\begin{align}
			&\ll \sum_{Q < q \ll X^{1/8}} \tau_4(q) \sum_{\substack{\alpha, \beta \\ [\alpha, \beta]^6 (\alpha, q)^{-4} \ll X/q^8}} h_3(\beta) \nonumber \\
			&\ll \sum_{Q < q \ll X^{1/8}} \tau_4(q) \sum_{\alpha_1 \mid q} \sum_{\beta^6 \ll \frac{X}{q^8}} h_3(\beta) \nonumber
			\sum_{\alpha_2 \mid \beta} \sum_{\substack{\alpha^6 \ll X \alpha_1^{4} \alpha_2^6 q^{-8} \beta^{-6} \\ \alpha_1 \alpha_2 \mid \alpha}} 1 \nonumber,
			\end{align}
where $\alpha_1 = (\alpha, q)$ and $\alpha_2 = (\alpha, \beta)$. Applying \eqref{eqn:ckps-bound}, we see that this is
			\begin{align}
			&\ll \sum_{Q < q \ll X^{1/8}} \tau_4(q) \sum_{\alpha_1 \mid q} \frac{X^{1/6}}{\alpha_1^{1/3} q^{4/3}} \sum_{\beta^6 \ll \frac{X}{q^8}} \frac{h_3(\beta) \tau(\beta)}{|\beta|} \label{eq:bound1} \\
			&\ll \sum_{Q < q \ll X^{1/8}} \tau_4(q) \sum_{\alpha_1 \mid q} \frac{X^{1/6} (\log X)^2}{\alpha_1^{1/3} q^{4/3}} \nonumber \\
			&\ll \sum_{Q < q \ll X^{1/8}} \tau_5(q) \frac{X^{1/6} (\log X)^2}{q^{4/3}} \nonumber \\
			&\ll \frac{X^{1/6} (\log X)^2 (\log \log X)^4}{Q^{1/3}} \nonumber,
					\end{align}
					as desired. 
\end{proof}

\begin{lemma} 
\label{lem:d-bound}
	The number of $K \in \mathcal{F}(X;G)$ with $|\beta| < \exp(A \sqrt{\log X})$ or $|\beta| > X^{1/6} / (\log X)^A$ for any fixed $A >0$, is 
		\[
			\ll_{A} X^{1/6} (\log X) (\log\log X).
		\]
\end{lemma}

\begin{proof}
 The proof is identical to that of Lemma \ref{lem:bounded-q2-q3} through \eqref{eq:bound1}, apart from the ranges of $q$ (which we no longer restrict) 
 and $\beta$. As each inner $\beta$-sum in the analogue of \eqref{eq:bound1} is now
 $\ll (\log X) (\log \log X)$, our sum is
 \[
 	\ll \sum_{q = 1}^{\infty} \tau_4(q) \sum_{\alpha_1 \mid q} \frac{X^{1/6} (\log X) (\log \log X)}{\alpha_1^{1/3} q^{4/3}} \ll X^{1/6} (\log X) (\log \log X),
	\]
as desired.
\end{proof}

\subsection{The main term}
We continue to observe the conventions of Notation \ref{ournotation}, and also the convention that all sums over $\beta'$ will be over the range 
$\exp(\sqrt{\log X}) \leq \beta' \leq X^{1/6} (\log X)^{-10}$.
Using Lemmas \ref{lem:extra_cont}, \ref{lem:bounded-q2-q3}, and \ref{lem:d-bound} with $Q = (\log X)^5$ and $A = 10$, we rewrite the sum in 
\eqref{eFXGE} as

\begin{align}
E(X) + \mathcal{F}(X; G, \Sigma) & =
\sum_{\beta'} \sum_{q \leq Q} \mathfrak{m}_\beta(q, \Sigma) \nonumber
\sum_{\substack{\alpha \\ \Delta(\Sigma) [\alpha', \beta']^6 q^8 (\alpha', q)^{-4} \leq X}} 1
\\ 
& = 
\sum_{\beta'} \sum_{q \leq Q} \nonumber
\mathfrak{m}_\beta(q, \Sigma) \tau(\beta')
\sum_{\substack{\alpha_1 \\ (\alpha_1, 6 \beta') = 1 \\ \Delta(\Sigma) (\alpha_1 \beta')^6 q^8 (\alpha_1, q)^{-4} \leq X}} 1 \\
& = \sum_{\beta'} \sum_{q \leq Q} \label{eq:initial_red}
\mathfrak{m}_\beta(q, \Sigma) \tau(\beta')
\sum_{q_1 \mid q}
\sum_{\substack{\alpha_3 \\ (\alpha_3, 6 \beta' q) = 1 \\ \Delta(\Sigma) (\alpha_3 \beta')^6 q^8 q_1^2 \leq X}} 1,
\end{align}
with $E(X) \ll X^{1/6} (\log X) (\log\log X)$,
seen as follows.
For each choice of $\beta'$, 
factor each $\alpha'$ as $\alpha' = \alpha_1 \alpha_2$ with $\alpha_2 = (\alpha', \beta')$. For each $\alpha_1$ there are $\tau(\beta')$ choices of $\alpha_2$ corresponding to 
each fixed value of $[\alpha', \beta']$, yielding the second line. Then, write $q_1 = (\alpha_1, q)$ and $\alpha_3 = \alpha_1 / q_1$ to conclude the third. 

The $\Sigma$-compatibility condition on $\alpha$ is equivalent to fixing its sign and fixing the congruence of $\alpha'$ in precisely $6$ invertible residue classes modulo $144$. We may thus apply Lemma \ref{lem:squarefree-sum} with $c=144$ and $m = \beta' q$ to see that the sum over $\alpha_3$ in \eqref{eq:initial_red} equals
$$ 
6 \cdot \frac{X^{1/6}}{\beta'{\Delta(\Sigma)}^{1/6} q^{4/3} q_1^{1/3}} \cdot \frac{6\psi(6\beta' q)}{\pi^2 \phi(144)} + 
O\left(\frac{X^{1/12} \tau(\beta q)}{|\beta|^{1/2} q^{2/3} q_1^{1/6}}\right).
$$ 
By Lemma \ref{lem:easy-bound} the contribution of the error term towards \eqref{eq:initial_red} is
\begin{align*}
\ll \ & X^{1/12}
\sum_{\substack{\beta'}} 
\frac{\mu^2(\beta) \tau(\beta)^2 h_3(\beta)}{|\beta|^{1/2}}
\sum_{\substack{q \leq Q}} \frac{\tau_4(q) \tau(q)^2}{q^{2/3}} \\
\ll \ & X^{1/12} \cdot X^{1/12} (\log X)^{-2} \cdot Q^{1/3 + \epsilon} \ll X^{1/6},
\end{align*}
where to bound the sum over $\beta$ we note that $\tau(b)^2 \leq \tau_4(\beta)$ for squarefree $\beta$, and apply 
\eqref{eqn:ckps-bound} with partial summation. We conclude that 
\begin{align}
\mathcal{F}(X; G, \Sigma) & = 
\frac{3 X^{1/6}}{8 \pi^2 \Delta(\Sigma)^{1/6}} 
\sum_{\beta'} \frac{ \tau(\beta') \psi(\beta')}{\beta'} \sum_{q \leq Q} \label{eq:truncated}
\mathfrak{m}_\beta(q, \Sigma) \frac{\psi(q)}{q^{4/3}} \prod_{p \mid q} \Big( 1 + \frac{1}{p^{1/3}} \Big) \\ & + O\big(X^{1/6} (\log X) (\log\log X)\big). \nonumber
\end{align}
To proceed, write $\psi(\beta') = \sum_{e \mid \beta'} \frac{\mu(e)}{\sigma(e)}$.
Using Lemma \ref{lem:easy-bound}, we see that the contribution to \eqref{eq:truncated} from $e > E$ is
\begin{align*}
\ll \ & X^{1/6} \sum_{\substack{\beta'}} \frac{\tau(\beta') h_3(\beta)}{\beta'} \sum_{\substack{e > E \\ e \mid \beta'}} \frac{1}{\sigma(e)} \sum_{q \leq Q}
 \frac{\psi(q) \tau_4(q)}{q^{4/3 - \epsilon}} \\
\ll \ & X^{1/6} \sum_{\substack{\beta'}} \frac{\tau(\beta') h_3(\beta)}{\beta'} \cdot \frac{\tau(\beta)}{E} \cdot O(1) \\
\ll \ & \frac{X^{1/6} (\log X)^4}{E},
\end{align*}
so that with $E := (\log X)^4$ we have
\begin{align}
\mathcal{F}(X; G, \Sigma) & = 
\frac{3 X^{1/6}}{8 \pi^2 \Delta(\Sigma)^{1/6}} 
\sum_{e \leq E} \frac{\mu(e)}{\sigma(e)} \sum_{q \leq Q} \frac{\psi(q)}{q^{4/3}} \prod_{p \mid q} \Big( 1 + \frac{1}{p^{1/3}} \Big)
\sum_{\substack{\beta' \\ e \mid \beta'}}
\mathfrak{m}_\beta(q, \Sigma) \frac{\tau(\beta')}{\beta'} \nonumber \\ & + O(X^{1/6}), \label{eq:fg2}
\end{align}
where the sums over $e$ and $q$ are supported on squarefree integers coprime to $6$.

By Theorem \ref{tmAverage} and partial summation, we have that
\begin{multline*}
\sum_{\substack{\beta' \\ e \mid \beta'}}
\mathfrak{m}_\beta(q, \Sigma) \frac{\tau(\beta')}{\beta'} = \frac{91 \mu^2(6eq) C(q, e) \tau(e) w_2 w_3 w_\infty (\log X)^2}{36 \cdot 72 \cdot \zeta(2)} \cdot \prod_{p \mid e}
\bigg( \frac{1}{p + 1} \bigg) \prod_{p \mid q} \bigg( \frac{p}{p + 1} \bigg) \\
+ O(\tau_4(q) \log X (\log \log X)^{7/2}),
\end{multline*}
with
\[
C(q, e) := \prod_p \bigg( 1 + \frac{\mathbf{1}_{p \nmid 6qe}}{p + 1} \bigg) \bigg( 1 - \frac{1}{p} \bigg).
\]
Here the factor $72$ in the denominator comes from the integral
$$
\int_{\exp(\sqrt{\log X})}^{X^{1/6} (\log X)^{-10}} \frac{\log t}{t} \ \textnormal{dt} = \frac{(\log X)^2}{72} + O\big((\log X)(\log \log X)\big).
$$
Substituting into \eqref{eq:fg2}, we check that the error contributes
$
O(X^{1/6} \log X (\log \log X)^{9/2}),
$
and that the main term, if summed over $e > E$ and $q > Q$, contributes $O(X^{1/6} (\log X)^2 (E^{-1} + Q^{-1/3 + \epsilon}))$, which is
$O(X^{1/6} \log X)$.

Summing over $\Sigma$, using that $\zeta(2) = \pi^2/6$ and $\frac{3}{8} \cdot \frac{91}{36 \cdot 72} \cdot 6 = \frac{91}{1152}$ and recalling the additional factor $1/12$ from equation \eqref{eInitial}, we obtain that 
\[
\# \mathcal{F}(X; G) = C_1 C_2 C_3 X^{1/6} (\log X)^2 + O\big( X^{1/6} (\log X) (\log\log X)^{9/2} \big),
\]
\[
C_1 := \frac{91}{13824 \pi^4} \prod_p \bigg( 1 + \frac{1}{p + 1} \bigg) \bigg( 1 - \frac{1}{p} \bigg),
\]
\begin{align*}
C_2 &:= \sum_{e, q} \mu^2(6eq) \frac{\mu(e) \tau(e)}{\sigma(e)} \prod_{p \mid e} \bigg( \frac{1}{p + 1} \bigg) \frac{\psi(q)}{q^{4/3}} \prod_{p \mid q} \Big( 1 + \frac{1}{p^{1/3}} \Big) \bigg( \frac{p}{p + 1} \bigg) \prod_{p \mid 6eq} \left(1 + \frac{1}{p + 1}\right)^{-1} \\
&:= \frac{3}{5} \prod_{p > 3} \left(1 - \frac{2}{(p + 1)(p + 2)} + \frac{1}{p^{4/3}} \left(1 + \frac{1}{p^{1/3}}\right) \frac{p^2}{(p + 1) (p + 2)}\right),
\end{align*}
\[
C_3 := \sum_{\Sigma} \frac{w_2 w_3 w_\infty}{\Delta(\Sigma)^{1/6}}.
\]

\subsection{Simplifications and comparison with the conjectural asymptotic}
\label{ssLS1}
	With the aim of comparing our proven asymptotic with the general conjectural formula of Loughran--Santens \cite{LS}, we now simplify the constant in our asymptotic for $\#\mathcal{F}(X;G)$. We begin by observing that
		\begin{align*}
			C_1
				&= \frac{91}{2^9 3^3 \pi^4} \prod_p \left(1+\frac{1}{p+1}\right)\left(1-\frac{1}{p}\right) 
				= \frac{91}{2^{11} 3^5} \prod_p \left(1+\frac{1}{p+1}\right)\left(1-\frac{1}{p}\right)\left(1-\frac{1}{p^2}\right)^2 \\
				&= \frac{91}{2^{11} 3^5} \prod_p \left(1-\frac{1}{p}\right)^3 \left( \frac{(p+2)(p+1)}{p^2}\right) 
				= \frac{91 \cdot 5}{2^9 3^9} \prod_{p \geq 5} \left(1-\frac{1}{p}\right)^3 \left( \frac{(p+2)(p+1)}{p^2}\right).
		\end{align*}
	As $C_2$ may be written as
		\[
			C_2 
				= \frac{3}{5} \prod_{p \geq 5} \left(\frac{p^2}{(p+2)(p+1)}\right) \left(1 + \frac{3}{p} + \frac{1}{p^{4/3}} + \frac{1}{p^{5/3}}\right),
		\]
	it therefore follows that
		\[
			C_1 C_2
				= \frac{91}{2^9 3^8} \prod_{p \geq 5} \left(1-\frac{1}{p}\right)^3\left(1 + \frac{3}{p} + \frac{1}{p^{4/3}} + \frac{1}{p^{5/3}}\right).
		\]
	This expression will be convenient in the comparison with the predicted constant, so it only remains to explicitly evaluate (and simplify where possible) the constant $C_3$. To this end, we begin by noting that the expression given is multiplicative across the three places $2,3,\infty$, so $C_3$ will be given as the product of the corresponding local summations. Second, the contribution from the archimedean place will be $2$, since there is no contribution from the discriminant (either in $C_3$, or in the definition of the weight $w_\infty$, and we have to then sum over the two quadratic \'etale algebras). The contribution from the $2$-adic and $3$-adic algebras may be determined by using the tables of local fields presented in the LMFDB \cite{LMFDB}, and which we compute in \verb^Magma^ \cite{Magma} with code available at \cite{github}. There are $27$ relevant \'etale algebras over $\mathbb{Q}_2$ and $33$ relevant algebras over $\mathbb{Q}_3$, and the net result is that 
		\begin{align*}
			C_3 
				&= 2 \left( \frac{23}{7} + \frac{8}{7} \cdot \frac{1}{2^{4/3}} + \frac{8}{7}\cdot \frac{1}{2^{5/3}}\right) \left( \frac{112}{39} + \frac{72}{39} \cdot \frac{1}{3^{4/3}} + \frac{72}{39} \cdot \frac{1}{3^{5/3}} \right) \\
				&= 2 \cdot \frac{8}{7} \cdot \frac{24}{13} \cdot \left( \frac{23}{8} + \frac{1}{2^{4/3}} + \frac{1}{2^{5/3}}\right) \left(\frac{14}{9} + \frac{1}{3^{4/3}} + \frac{1}{3^{5/3}}\right).
		\end{align*}
	Thus, all told, we find that
		\begin{multline*}
			C_1 C_2 C_3
				= \frac{1}{2^2 3^7} \left( \frac{23}{8} + \frac{1}{2^{4/3}} + \frac{1}{2^{5/3}}\right)\left(\frac{14}{9} + \frac{1}{3^{4/3}} + \frac{1}{3^{5/3}}\right) \times \\ 
				\prod_{p \geq 5} \left(1-\frac{1}{p}\right)^3\left(1 + \frac{3}{p} + \frac{1}{p^{4/3}} + \frac{1}{p^{5/3}}\right).
		\end{multline*}
	
	Loughran and Santens \cite{LS} have recently substantially clarified Malle's conjecture by providing an explicit prediction for the leading constant in the asymptotic. We do not aim for a fully general or independent discussion of their work, instead targeting a streamlined presentation that facilitates the comparison with the proven asymptotic above. Note that $G = D_6$ is generated by its elements of order $2$, $|\mathrm{Aut}(G)| = |G| = 12$, $G$ has four rational characters, and (with the usual conventions) $a(G) = 1/6$ and $b(G) = 3$. Noting that \cite[Lemma 6.38]{LS} and that \cite[Remark 3.5 and Lemma 6.32]{LS} and that all characters of $D_6$ are rational, Loughran and Santens's Conjecture 1.3.(1) specializes to predict that
		$$
			\#\{ K \in \mathcal{F}(X;G) : \mathbb{Q}(\zeta_3) \cap K = \mathbb{Q}\}
				\sim c(G)\cdot X^{1/6} (\log X)^{2},
		$$
	where $c(G)$ is given explicitly by
		\begin{align} \label{eqn:LS-constant}
			c(G)
				&= \frac{1}{6^2 \cdot 4 \cdot 2!} \cdot \frac{|\mathrm{Hom}(C_2,G)|}{|G|} \cdot \prod_p \left( \frac{(1-p^{-1})^{3}}{|G|} \sum_{\varphi_p \in \mathrm{Hom}(G_{\mathbb{Q}_p},G)} \frac{1}{p^{v_p(\Delta_{\varphi_p})/6}}\right) \notag \\
				&= \frac{1}{2^4 3^3} \prod_p \left( \frac{(1-p^{-1})^{3}}{|G|} \sum_{\varphi_p \in \mathrm{Hom}(G_{\mathbb{Q}_p},G)} \frac{1}{p^{v_p(\Delta_{\varphi_p})/6}}\right)
		\end{align}
	where we explain additional notation as needed. We will show at the end of this proof that 
		\begin{equation}\label{eqn:contains-mu3}
				\#\{ K \in \mathcal{F}(X;G) : K \cap \mathbb{Q}(\zeta_3) = \mathbb{Q}(\zeta_3)\}
					 = O(X^{1/6}),
		\end{equation}
	so our goal is to establish the equality $c(G) = C_1C_2C_3$. 
	This evidently amounts to evaluating the local terms in the Euler product above.
	
	An element $\varphi_p \in \mathrm{Hom}(G_{\mathbb{Q}_p},G)$ gives rise to a Galois degree $12$ \'etale algebra $A$ over $\mathbb{Q}_p$, given explicitly as $A = (\overline{\mathbb{Q}_p}^{\ker \varphi_p})^{ [ G : \mathrm{im}(\varphi_p)]}$. Let $D := \mathrm{im}(\varphi_p)$. The field $K_p := \overline{\mathbb{Q}_p}^{\ker \varphi_p}$ is then a Galois extension of $\mathbb{Q}_p$ with Galois group $D$, and $v_p(\Delta_{\varphi_p})$ is defined to be $v_p( \mathrm{Disc}(A/\mathbb{Q}_p)) = v_p( \mathrm{Disc}(K_p)) \cdot [G:D]$. We therefore find that
		\begin{equation} \label{eqn:LS-field-sum}
			\sum_{\varphi_p \in \mathrm{Hom}(G_{\mathbb{Q}_p},G)} \frac{1}{p^{v_p(\Delta_{\varphi_p})/6}}
				= \sum_{D \leq G} |\mathrm{Aut}(D)| \sum_{\substack{ K_p/\mathbb{Q}_p \text{ Galois} \\ \mathrm{Gal}(K_p/\mathbb{Q}_p) \cong D}} |\mathrm{Disc} (K_p)|_p^{\frac{[G:D]}{6}}.
		\end{equation}
	The subgroups of $G$ are, up to isomorphism, $1$, $C_2$, $C_3$, $C_2 \times C_2$, $C_6$, $S_3$, and $G$ itself, with isomorphism classes of size $1$, $7$, $1$, $3$, $1$, $2$, and $1$, respectively. (Note that the summation over $D$ in \eqref{eqn:LS-field-sum} is over subgroups, not isomorphism classes.) We now evaluate the sum over local fields on the right-hand side (an alternative way is to apply \cite[Corollary 8.11]{LS}).
	
	Suppose first that $p \equiv 1 \pmod{3}$. Then there are no $S_3$-extensions of $\mathbb{Q}_p$, and thus also no $G$-extensions either. There are three quadratic extensions of $\mathbb{Q}_p$ (one unramified and two with discriminant valuation $1$), four $C_3$-extensions (one unramified, three ramified with discriminant valuation $2$), one $C_2 \times C_2$-extension (with discriminant valuation $2$), and twelve $C_6$-extensions (one unramified, two with discriminant valuation $3$, three with discriminant valuation $4$, and six with discriminant valuation $5$). All told, if $p \equiv 1 \pmod{3}$, then bookkeeping reveals
		\begin{equation}\label{eqn:LS-accounting-tame}
			\sum_{D \leq G} |\mathrm{Aut}(D)| \sum_{\substack{ K_p/\mathbb{Q}_p \text{ Galois} \\ \mathrm{Gal}(K_p/\mathbb{Q}_p) \cong D}} |\mathrm{Disc} (K_p)|_p^{\frac{[G:D]}{6}}
				= 12 + \frac{36}{p} + \frac{12}{p^{4/3}} + \frac{12}{p^{5/3}}.
		\end{equation}
	If instead $p \equiv 2 \pmod{3}$ but $p \ne 2$, then there is a unique $S_3$-extension of $\mathbb{Q}_p$ (with discriminant valuation $4$) and a unique $G$-extension (with discriminant valuation $10$), but there are no ramified $C_3$-extensions, and only two ramified $C_6$-extensions (both with discriminant valuation $3$). The accounting of $C_2$ and $C_2 \times C_2$-extensions is the same as when $p \equiv 1 \pmod{3}$, and we find in total that \eqref{eqn:LS-accounting-tame} holds also if $p \equiv 2\pmod{3}$.
	
	For $p = 2$ and $p=3$, we enumerate over the $27$ and $33$ relevant \'etale algebras found in the LMFDB, computing the summation on the right-hand side of \eqref{eqn:LS-field-sum} explicitly. We find if $p=2$ that
		\[
			\sum_{D \leq G} |\mathrm{Aut}(D)| \sum_{\substack{ K_2/\mathbb{Q}_2 \text{ Galois} \\ \mathrm{Gal}(K_2/\mathbb{Q}_2) \cong D}} |\mathrm{Disc} (K_2)|_2^{\frac{[G:D]}{6}}
				= \frac{69}{2} + \frac{12}{2^{4/3}} + \frac{12}{2^{5/3}}
		\]
	and if $p=3$ that
		\[
			\sum_{D \leq G} |\mathrm{Aut}(D)| \sum_{\substack{ K_3/\mathbb{Q}_3 \text{ Galois} \\ \mathrm{Gal}(K_3/\mathbb{Q}_3) \cong D}} |\mathrm{Disc} (K_3)|_3^{\frac{[G:D]}{6}}
				= \frac{224}{9} + \frac{16}{3^{4/3}} + \frac{16}{3^{5/3}}.
		\]
	Returning to \eqref{eqn:LS-constant}, we therefore find that
		\begin{align*}
			c(G)
				&= \frac{1}{2^4 3^3} \cdot \frac{1}{8} \cdot \left(\frac{23}{8} + \frac{1}{2^{4/3}} + \frac{1}{2^{5/3}}\right) \cdot \frac{8}{27} \cdot \left(\frac{56}{27} + \frac{4}{3}\cdot \frac{1}{3^{4/3}} + \frac{4}{3}\cdot \frac{1}{3^{5/3}}\right) \cdot \\
				&\hspace{3.2in} \cdot \prod_{p \geq 5} \left(1-\frac{1}{p}\right)^3 \left(1 + \frac{3}{p} + \frac{1}{p^{4/3}} + \frac{1}{p^{5/3}}\right) \\
				&= \frac{1}{2^2 3^7} \left( \frac{23}{8} + \frac{1}{2^{4/3}} + \frac{1}{2^{5/3}}\right)\left(\frac{14}{9} + \frac{1}{3^{4/3}} + \frac{1}{3^{5/3}}\right)\prod_{p \geq 5} \left(1-\frac{1}{p}\right)^3\left(1 + \frac{3}{p} + \frac{1}{p^{4/3}} + \frac{1}{p^{5/3}}\right) \\
				&= C_1 C_2 C_3.
		\end{align*}
	
	\subsection{\texorpdfstring{$D_6$-extensions with a given quadratic subfield}{D6-extensions with a given quadratic subfield}}
	\label{ssLS2}
	The following proposition suffices to establish \eqref{eqn:contains-mu3} by taking $d=-3$. It also follows from more general work of Alberts \cite{Alberts}.
	
	\begin{proposition}
		Let $d \ne 1$ be a squarefree integer. Then as $X \to \infty$,
			\[
				\#\{K \in \mathcal{F}(X;G) : \mathbb{Q}(\sqrt{d}) \hookrightarrow K\}
					= O_d(X^{1/6}).
			\]
	\end{proposition}
	
	\begin{proof}
		If $\mathbb{Q}(\sqrt{d}) \hookrightarrow K$ for some $K \in \mathcal{F}(X;G)$, then there are two fundamentally different possibilities: either $\mathbb{Q}(\sqrt{d})$ is the subfield of $K$ fixed by one of the two subgroups of $D_6$ isomorphic to $S_3$ (in which case $K$ is the compositum of $\mathbb{Q}(\sqrt{d})$ with an $S_3$-extension $L/\mathbb{Q}$), or $\mathbb{Q}(\sqrt{d})$ is the subfield of $K$ fixed by the unique subgroup of $G$ isomorphic to $C_6$. We will show that each of these possibilities in turn contribute an amount that is $O_d(X^{1/6})$.
		
		In the first case, it suffices to consider those $K$ obtained as the compositum of $\mathbb{Q}(\sqrt{d})$ with a Galois $S_3$-extension $L/\mathbb{Q}$. However, in this case, we have $|\mathrm{Disc}(L)|^2 \leq |\mathrm{Disc}(K)|$, so the number of such $K$ is at most the number of $L$ for which $|\mathrm{Disc}(L)| \leq X^{1/2}$. By \cite{BhargavaWood}, the number of such fields $L$ is $O(X^{1/6})$, so the result follows in this case.
		
		For the second case, $K$ may be obtained as the normal closure of the compositum of a quadratic extension $\mathbb{Q}(\sqrt{\alpha})$ with a $S_3$-cubic field $L$ whose discriminant satisfies $\mathrm{Disc}(L) = d q^2$, where $\alpha$ and $q$ are squarefree integers, except possibly at $2$ and $3$. From the considerations in the start of this section, we then find $|\mathrm{Disc}(K)| \gg_d \alpha_1^6 q^8 \alpha_2^{2}$ where $\alpha_2 = (\alpha,q)$ and $\alpha_1 = \alpha/\alpha_2$. We also define $m_d(q)$ to be the sum over all local specifications $\Sigma$ of $m_d(q, \Sigma)$. There is then a constant $c_d>0$ so that number of fields $K$ arising in this second case is therefore at most
			\begin{align*}
				\sum_{q \leq X^{1/8}} m_d(q) \sum_{\alpha_2 \mid q} \sum_{\alpha_1 \leq c_d X^{1/6}/q^{-4/3} \alpha_2^{-1/3}} 1
					&\ll_d X^{1/6} \sum_{q \leq X^{1/8}} \frac{m_d(q)}{q^{4/3}} \prod_{p \mid q} \left(1 + \frac{1}{p^{1/3}}\right) \\
					&\ll X^{1/6} \sum_{q \leq X^{1/8}} \frac{m_d(q) \tau(q)}{q^{4/3}} \\
					&\ll X^{1/6},
			\end{align*}
		completing the proof of the proposition.
	\end{proof}
	
	\begin{remark}
		It should not be difficult to extend the above proposition into an asymptotic formula. For fields obtainable in the first case, it suffices to understand asymptotics for Galois $S_3$-extensions with finitely many local conditions, which follows from \cite{BhargavaWood}. For fields in the second case, because the summation over $q$ converges, it suffices to understand only the distribution of quadratic fields with finitely many local conditions, as follows from \cite{DW3}.
	\end{remark}

\end{document}